\documentclass[a4,12pt,reqno]{amsart}
\usepackage[english]{babel}
\usepackage{graphicx} 
\usepackage{amssymb,amstext,amscd,amsfonts,mathdots,mathrsfs,amsthm,amsmath}
\usepackage{enumerate}
\usepackage[all,cmtip]{xy} 
\usepackage[normalem]{ulem}
\usepackage{color}
\usepackage{tikz} 
\usepackage{tikz-cd} 

\usetikzlibrary{calc}
\usetikzlibrary{matrix,arrows,decorations.pathmorphing}
\usetikzlibrary{patterns, arrows,decorations.markings}
\usetikzlibrary{positioning,fit}
\usetikzlibrary{shapes.misc}

\definecolor{mygreen}{RGB}{0, 180, 0}
\definecolor{myred}{RGB}{200, 0, 0}
\definecolor{arcblue}{RGB}{0, 0, 255}
\definecolor{arrowcolor}{RGB}{0, 0, 0}

\usepackage[colorlinks=true,pagebackref,hyperindex]{hyperref}

\definecolor{dark-green}{RGB}{14,150,2}
\newcommand\myshade{85}
\colorlet{mylinkcolor}{violet}
\colorlet{mycitecolor}{red}
\colorlet{myurlcolor}{cyan}

\hypersetup{
  linkcolor  = mylinkcolor!\myshade!black,
  citecolor  = mycitecolor!\myshade!black,
  urlcolor   = myurlcolor!\myshade!black,
  colorlinks = true,
}

\newtheorem{Thm}{Theorem}[section]

\newtheorem{Lem}[Thm]{Lemma}
\newtheorem{Cor}[Thm]{Corollary}

\newtheorem{Prop}[Thm]{Proposition}
\newtheorem{Prop-Def}[Thm]{Proposition-Definition}
\newtheorem*{Conj*}{Conjecture}
\newtheorem{Que}[Thm]{Question} 

\theoremstyle{definition}
\newtheorem{Ex}[Thm]{Example}
\newtheorem{Def}[Thm]{Definition}

\newtheorem{Rem}[Thm]{Remark}

\newcommand{\rpoint}{\color{red}{\bullet}}

\definecolor{arcblue}{RGB}{0, 0, 255} 
\newcommand{\bpoint}{\color{arcblue}{\bullet}}
\newcommand{\silt}{\mathbf{silt}\hspace{.05in}} 

\newcommand{\stautilt}{\mbox{\bf s$\tau$-tilt}\hspace{.02in}}
\newcommand{\brick}{\mathbf{brick}\hspace{.02in}}
\newcommand{\fbrick}{\mathbf{fbrick}\hspace{.02in}}

\newcommand{\D}{\mathcal D}
\newcommand{\T}{\mathcal T}
\newcommand{\Z}{{\mathbb Z}}
\newcommand{\N}{{\mathbb N}}

\newcommand{\R}{\mathbf{R}}

\newcommand{\za}{\alpha}
\newcommand{\zb}{\beta}
\newcommand{\zd}{\delta}
\newcommand{\zD}{\Delta}

\newcommand{\zg}{\gamma}

\newcommand{\op}{\oplus}

\newcommand{\ot}{\otimes}

\DeclareMathOperator{\moduleCategory}{{\mathsf{mod}}} 
\renewcommand{\mod}{\moduleCategory} 
\renewcommand{\dim}{{\rm dim}}       

\DeclareMathOperator{\thick}{\mathsf{thick}}
\DeclareMathOperator{\per}{\mathsf{per}}

\newcommand{\bb}{\mathrm{b}}
\newcommand{\p}{\mathrm{p}}

\newcommand{\h}{{\mathrm H}}

\newcommand{\Hom}{\operatorname{Hom}\nolimits}

\newcommand{\rad}{\operatorname{rad}\nolimits}

\newcommand{\End}{\operatorname{End}\nolimits}
\newcommand{\RHom}{\mathbf{R}\strut\kern-.2em\operatorname{Hom}\nolimits}
\newcommand{\RshHom}{\mathbf{R}\strut\kern-.2em\mathscr{H}\strut\kern-.3em\operatorname{om}\nolimits}
\newcommand{\shHom}{\mathscr{H}\strut\kern-.3em\operatorname{om}\nolimits}
\newcommand{\shEnd}{\mathscr{E}\strut\kern-.3em\operatorname{nd}\nolimits}

\numberwithin{equation}{section}

\title[$\tau$-tilting finiteness and silting-discreteness of (skew-) gentle algebras]{On the $\tau$-tilting finiteness and silting-discreteness of graded (skew-) gentle algebras}

\begin{document}
\setlength{\baselineskip}{16pt}

\author{Wen Chang}\address{Wen Chang, School of Mathematics and Statistics, Shaanxi Normal University, Xi'an 710062, China}\email{changwen161@163.com}

\author{Haibo Jin}
\address{Haibo Jin, School of Mathematical Sciences,  Key Laboratory of Mathematics and Engineering Applications (Ministry of Education),   Shanghai Key Laboratory of PMMP,
  East China Normal University,
 Shanghai 200241, China}
\email{hbjin@math.ecnu.edu.cn}

\author{Sibylle Schroll}\address{Sibylle Schroll, Insitut f\"ur Mathematik, Universit\"at zu K\"oln, Weyertal 86-90, K\"oln, Germany }\email{schroll@math.uni-koeln.de}

\author{Qi Wang}
\address{Qi Wang, School of Mathematical Sciences, Dalian University of Technology, Dalian, 116024, China}
\email{wang2025@dlut.edu.cn}


\begin{abstract}
This paper investigates finiteness conditions for gentle and skew-gentle algebras. First, we prove that a skew-gentle algebra is $\tau$-tilting finite if and only if it is representation-finite, which extends the result for gentle algebras in \cite{P19}. Second, using surface models, we characterize silting-discreteness for the perfect derived categories of graded gentle and skew-gentle algebras. Specifically, for a graded gentle algebra, silting-discreteness is equivalent to its associated surface being of genus zero with non-zero winding numbers for all simple closed curves. We further extend this geometric characterization to graded skew-gentle algebras via orbifold surface models.
\end{abstract}

\maketitle
\tableofcontents
\section{Introduction}

Module categories and derived categories, or more generally abelian categories and triangulated categories, play a fundamental role in many areas of mathematics, including representation theory, algebraic geometry, and algebraic topology. A fruitful approach to understanding these categories is to study special classes of generators, such as tilting, $\tau$-tilting, and silting objects. Tilting objects play a central role in the study of equivalences between these categories \cite{APR79,BGP73,BB80,HR82}. As natural generalizations of tilting objects, $\tau$-tilting modules, introduced for module categories in \cite{AIR14}, and silting complexes, introduced for derived categories in \cite{KV88}, have emerged as central objects of interest in modern representation theory. These objects are important in studying equivalences of categories as well as mutations; see, for example, \cite{Rickard89,Keller94,Ke98}. Moreover, there are deep connections with several other fundamental structures in mathematics, including Bridgeland stability conditions, cluster theory, torsion theory, and $t$-structures; see, for example, \cite{AI12,AIR14,IY,KoY,QW} for details.

A central topic in representation theory is the study of suitable finiteness conditions for module categories and derived categories, in particular, from the perspective of the aforementioned special classes of objects. For instance, representation-finiteness \cite{G72,DR73} reflects the complexity of the module category; $\tau$-tilting finiteness \cite{AIR14,DIJ19} captures the behavior of wide subcategories and torsion classes in module categories; silting-discreteness \cite{Aihara13} governs the structure of silting objects and their mutations in triangulated categories \cite{AI12,AHMW21,AH24,AMY}. These finiteness conditions are closely related to each other and interact substantially \cite{AM17,BPP,P19}.

However, determining whether a given algebra satisfies these finiteness conditions remains a challenging problem in general. In this paper, we investigate these finiteness conditions in module categories and derived categories for two important classes of algebras, namely, gentle algebras introduced in \cite{AH81,AS-tilting-cotilting-equ} and their generalizations, skew-gentle algebras first defined in \cite{GeissdelaPena}. The representation theory of these algebras has been extensively studied, and both their module categories and derived categories admit rich and explicit combinatorial structures. Recent developments have shown that geometric surface models provide an effective framework for describing these structures for gentle algebras; see, for example, \cite{B16, HKK17, OPS18, LP20, BC21, CJS22, JSW23, C25}, as well as for skew-gentle algebras; see, for example, \cite{AB22, LSV22, QZZ22, HZZ23, AP24, CK24, BSW24}. More precisely, indecomposable objects in their module and derived categories can be interpreted as curves on a marked surface, while morphisms correspond to intersections of curves, Auslander--Reiten translations are realized by suitable rotations, and certain categorical reductions can be described by cutting the surface. This geometric perspective has proved to be particularly powerful in studying homological and representation-theoretic properties of (skew-)gentle algebras.

A finite-dimensional algebra is said to be \emph{$\tau$-tilting finite} (see Definition~\ref{def::tau-tilt-finite}) if it admits only finitely many $\tau$-tilting modules up to isomorphism. Our first main result provides a characterization of $\tau$-tilting finiteness for skew-gentle algebras, extending and generalizing the result for gentle algebras established in \cite{P19}.

\begin{theorem}[Theorem \ref{theo::main-result-skew-gentle}]\label{thm::1.1}
Let $A$ be a skew-gentle algebra. Then, $A$ is $\tau$-tilting finite if and only if $A$ is representation-finite, or equivalently, $A$ admits no band module.
\end{theorem}

To prove Theorem \ref{thm::1.1}, we make use of the complete classification of indecomposable modules over a skew-gentle algebra in terms of asymmetric and symmetric strings and bands (Section \ref{section:skew-gentle}). This classification was first obtained in \cite{CB-skew-gentle} for the local algebra $\Bbbk \langle \epsilon, \alpha \mid \epsilon^2=\epsilon,\ \alpha^2=0 \rangle$, and was subsequently extended to arbitrary clannish algebras in \cite{CB-skew-gentle-2}, see also  \cite{B91,Deng00,Hansper-thesis}.

A triangulated category is said to be \emph{silting-discrete} (see Definition~\ref{Def:silting-discrete}) if there are only finitely many silting objects between any two given silting objects with respect to a natural partial order (see Section~\ref{Section:silting-discrete}). There is a bijection between support $\tau$-tilting modules and $2$-term silting objects  \cite[Theorem~3.2]{AIR14}, and silting-discreteness is a stronger property than $\tau$-tilting finiteness. However, these two notions are closely intertwined: silting-discreteness can be detected through the $\tau$-tilting finiteness of endomorphism algebras of silting objects, as shown in \cite[Theorem~2.4]{AM17}.

For a graded marked surface $(S,\eta)$, where $S$ is a marked surface and $\eta$ is a line field on $S$, it is shown in \cite{HKK17} that the partially wrapped Fukaya category of $(S,\eta)$ is equivalent to the derived category of a graded gentle algebra arising from a collection $\Delta$ of simple arcs, called a dissection. Conversely, it is shown in \cite{LP20} that for any homologically smooth graded gentle algebra $A$, one can construct a   graded marked surface $(S,\eta)$ such that the partially wrapped Fukaya category of $(S, \eta)$ is equivalent to the perfect derived category of $A$. Our second main result provides a complete classification of silting-discrete graded gentle algebras in terms of certain invariants of the associated surface.

\begin{theorem}[Theorem \ref{thm:main1} and Corollary \ref{Cor:siltingdiscretegentle}]\label{thm::1.2}
    Let $A$ be a homologically smooth, proper graded gentle algebra and let $(S,\eta,\Delta)$ be the associated graded surface with boundary components $\partial_i$, for $1 \leq i \leq b$. Then the following are equivalent:
    \begin{enumerate}
        \item The perfect derived category $\per(A)$ is silting-discrete.
        \item $S$ is of genus $0$ and admits no simple closed curves with winding number $0$.
        \item $S$ is of genus $0$ and the following `No Equipartition' condition is satisfied
        \[ \sum_{j\in J}m_j\not=\frac{1}{2}\sum_{j=1}^bm_j  \] 
        for all $J\subset\{1,2,\ldots,b\}$,
        where  $m_i=\omega_\eta(\partial_i)+2$ for $1\le i\le b$. 
    \end{enumerate}
\end{theorem}

One notable property of a silting-discrete triangulated category shown in \cite{AM17}
is that any pre-silting object is partial silting (see Proposition \ref{Prop:silt}). Our approach to Theorem \ref{thm::1.2} relies on a complete characterization of this property given in \cite{JSW23} for graded gentle algebras. More precisely, it is shown in \cite{JSW23} that any pre-silting object is partial silting in the derived category of a graded gentle algebra if and only if the associated surface model is either of genus $0$, or of genus $1$ such that the winding numbers of simple closed curves satisfy certain special conditions (see Theorem \ref{JSW}). Combining this with Theorem \ref{thm::1.2}, we see that the property that every pre-silting object is partial silting is not sufficient to guarantee silting-discreteness.

Another important finiteness condition for an algebra is derived-discreteness \cite{V01}. It is well-established that a finite-dimensional algebra $A$ (not derived equivalent to a Dynkin quiver) is derived-discrete if and only if it is a gentle one-cycle algebra not satisfying the clock condition \cite{V01, BGS04}. In terms of the associated surface model, this characterization can be reinterpreted as follows: $A$ is derived-discrete if and only if its associated surface is an annulus that admits no simple closed curves with winding number zero.
While the fact that derived-discreteness implies silting-discreteness is already known (cf. \cite[Proposition 6.12]{BPP16}, \cite{YY21}, and \cite{AH24}), as a direct application of Theorem \ref{thm::1.2}, our geometric framework provides another approach to this implication, see Corollary \ref{Cor:deriveddiscreteimpliessiltingdiscrete}.

Theorems \ref{thm::1.1} and \ref{thm::1.2} above enable us to study the silting-discreteness of graded skew-gentle algebras. Similarly to graded gentle algebras, we can also associate to any graded skew-gentle algebra $A$ a surface dissection $\Delta$ of a graded marked surface $(S,O,\eta,\zD)$, where $O$ is the set of orbifold points of order two corresponding to the special loops of $A$ (see \cite{LSV22, QZZ22, BSW24}). The last result of our paper is as follows.

\begin{theorem}[Theorem \ref{Thm:siltingdiscreteforskewgentle}]\label{Thm:Amain3}
Let $A$ be a graded skew-gentle algebra with surface model $(S,O,\eta,\Delta)$. Consider the following conditions.
  \begin{enumerate}
      \item $\per(A)$ is silting-discrete.
      \item $S$ is of genus $0$ and $\# O\le 1$. Moreover, $\omega_\eta(\gamma)\neq 0$ for any simple closed curve $\gamma$.
      \item $S$ is of genus $0$ and 
        \[ \sum_{j\in J}m_j\not=2 \mbox{ for all $J\subset\{1,2,\ldots,b\}$,} \] 
        where  $m_i=\omega_\eta(\partial_i)+2$ for $1\le i\le b$. 
 \end{enumerate}
 Then (1) implies (2), and (2) is equivalent to (3).
\end{theorem}
To prove Theorem \ref{Thm:Amain3}, we use a similar strategy
as for Theorem \ref{thm::1.2}. However, since the endomorphism algebra of a silting object in the derived category of a graded skew-gentle algebra may no longer be skew-gentle, we cannot directly apply Theorem \ref{thm::1.1} to prove the converse direction of the above theorem by showing that the endomorphism algebra of any silting object is $\tau$-tilting finite. That is, we cannot prove that condition $(2)$ implies condition $(1)$. Nevertheless, we still expect the three conditions in Theorem \ref{Thm:Amain3} to be equivalent. 

The paper is organized as follows. In Section 2, we recall background material on $\tau$-tilting theory, silting-discrete triangulated categories, graded gentle algebras and their surface models, as well as skew-gentle algebras. Section 3 is devoted to the proof of Theorem \ref{thm::1.1}. In Section 4, we study silting-discreteness for graded gentle algebras and prove Theorem \ref{thm::1.2}, dividing the proof into its necessary and sufficient parts. In Section 5, we  apply Theorems \ref{thm::1.1} and \ref{thm::1.2} to study silting-discreteness for graded skew-gentle algebras and prove Theorem \ref{Thm:Amain3}.

\section*{Acknowledgments}
The authors would like to thank Zhengfang Wang for many helpful discussions and for clarifying various questions regarding graded skew-gentle algebras. 

The first author is supported by the Fundamental Research Funds for the Central Universities (No.~GK202403003) and the NSF of China (No.~12271321). 
The second author is supported by the National Key R\&D Program of China (No.~2024YFA1013801) and the Science and Technology Commission of Shanghai Municipality (No.~22DZ2229014). 
The third author is supported by the DFG through the project SFB/TRR 191 (No.~281071066-TRR 191). 
The fourth author is supported by the NSF of China (No.~12401048) and the Fundamental Research Funds for the Central Universities (No.~DUT25RC(3)132).

\section{Preliminaries}

In this paper, all algebras will be assumed to be over a base field $\Bbbk$ which is algebraically closed, and all (differential graded) modules considered are right modules.

\subsection{$\tau$-tilting finiteness}\label{Section:tau-tilting}
Let $A$ be a finite-dimensional $\Bbbk$-algebra. We denote by $\mod A$ the category of finitely generated right $A$-modules, and $\tau$ denotes the Auslander-Reiten translation in $\mod A$. 

For any $M\in \mod A$, let $|M|$ be the number of isomorphism classes of indecomposable direct summands of $M$. Following \cite{AIR14}, $M$ is called \emph{$\tau$-rigid} if $\Hom_A(M,\tau M)=0$, and is called \emph{$\tau$-tilting} if $M$ is $\tau$-rigid and $|M|=|A|$.

\begin{Def}\label{def::tau-tilt-finite}
A finite-dimensional $\Bbbk$-algebra $A$ is called \emph{$\tau$-tilting finite} if it admits only finitely many $\tau$-tilting modules up to isomorphisms, and \emph{$\tau$-tilting infinite} otherwise.
\end{Def}

By \cite[Theorem 0.2]{AIR14}, every $\tau$-rigid $A$-module occurs as a direct summand of some $\tau$-tilting $A$-module. Consequently, $A$ is $\tau$-tilting finite if and only if $A$ has only finitely many pairwise non-isomorphic indecomposable $\tau$-rigid modules. There are several known criteria equivalent to $\tau$-tilting finiteness; in particular, we will use in this paper the so-called \emph{brick-$\tau$-rigid correspondence} as follows.

A right $A$-module $M$ is called a \emph{brick} if $\End_A(M)\simeq \Bbbk$. We say that $A$ is \emph{brick-finite} if it admits only finitely many non-isomorphic bricks, and \emph{brick-infinite} otherwise.
Let $\brick A$ be the set of bricks in $\mod A$, and $\fbrick A$ the set of bricks $M$ such that the smallest torsion class containing $M$ is functorially finite. 

\begin{Prop}[{\cite[Theorem 4.2]{DIJ19}}]\label{prop::brick-tau-rigid}
There exists a bijection sending indecomposable $\tau$-rigid $A$-modules in $\mod A$ to bricks in $\fbrick A$, by 
\begin{center}
$M\mapsto M/\rad_B(M), \quad B:=\End_A(M)$,
\end{center}
If $A$ is $\tau$-tilting finite, then $\fbrick A=\brick A$.
\end{Prop}

We are now able to see that $A$ is $\tau$-tilting finite if and only if it is brick-finite.

There is a larger set $\stautilt A$ of \emph{support $\tau$-tilting $A$-modules}, consisting of all $\tau$-tilting $\left ( A/A e A\right )$-modules for some idempotent $e$ of $A$. The most fascinating aspect of $\tau$-tilting theory is, in fact, built upon and further developed from the structure of $\stautilt A$. Since we will not study $\stautilt A$ directly, we omit the details, but refer to \cite{AIR14} and \cite{DIRRT}, to name just two, for more materials.

\subsection{Silting-discreteness}\label{Section:silting-discrete}
Let $\T$ be a triangulated $\Bbbk$-category and $P\in \T$. We denote by $\thick_{\T}(P)$ the smallest full triangulated subcategory of $\T$ that contains $P$ and is closed under taking direct summands. For a differential graded $\Bbbk$-algebra $A$, we denote by $\D(A)$ the \emph{derived category} of $A$, and by $\per(A):=\thick_{\D(A)}(A)$ the \emph{perfect derived category} of $A$. We denote by $\D^{\bb}(A)$ the full subcategory of $\D(A)$ consisting of the objects $M$ whose total cohomology is finite-dimensional (that is, $\sum_n \dim_{\Bbbk}\; \h^{n}(M)<\infty$). 

If $\per (A)\subset \D^{\bb}(A)$, we say that $A$ is \emph{proper}. Note that a differential graded algebra with a trivial differential is proper if and only if it is finite-dimensional. If $A\in \per (A\ot_{\Bbbk}A^{\rm op})$, we say that $A$ is \emph{homologically smooth}. Note that if $A$ is homologically smooth and proper, then we have $\per(A)=\D^{\rm b}(A)$. 

We recall from \cite[Definition 2.1]{AI12} that an object $P$ in $\T$ is called \emph{pre-silting} if $\Hom_{\T}(P, P[n])=0$ for all $n\ge 1$, and a pre-silting object $P$ is called \emph{silting} if $\thick_\T(P)=\T$. We denote by $\silt \T$ the set of isomorphism classes of silting objects in $\T$. In this paper, we always assume that $\silt \T$ is nonempty. 

There is a partial order on $\silt \T$ (see \cite[Theorem 2.11]{AI12}) defined as 
\[
P\geq Q \quad \text{ if }\quad \Hom_\T(P,Q[i])=0, \; (i>0)
\]
for $P,Q \in \silt \T$. 

For any finite-dimensional $\Bbbk$-algebra $A$, the perfect derived category $\per(A)$ is a triangulated category with a silting object $A$. 

For any $P\geq Q\in \silt \T$, let $[P,Q]$ be the set of isomorphism classes of silting objects $M$ such that $P\geq M \geq Q$. We note that $[P,Q]$ is again a poset under the restriction of the partial order of $\silt \T$. The notion of silting-discreteness was introduced as follows.

\begin{definition}[{\cite[Definition 3.6]{Aihara13}}]\label{Def:silting-discrete}
Let $\T$ be a triangulated category with silting objects. We say that $\T$ is \emph{silting-discrete} if for any $P, Q$ in $\silt \T$ with $P\geq Q$, the poset $[P, Q]$ is finite.
\end{definition}

There are many equivalent characterizations of silting-discreteness. We recall some of them in the following proposition, where the equivalence of (1), (2), and (3) is due to \cite[Proposition 3.8]{Aihara13}, the equivalence of (1) and (4) is established in \cite{AM17}, and the equivalence of (4) and (5) follows from \cite[Corollary 2.8]{DIJ19}. 

\begin{proposition}\label{Prop:equ-silting-discrete}
Let $\T$ be a Hom-finite Krull-Schmidt triangulated category with silting objects.
The following conditions are equivalent:
\begin{enumerate}
\item $\T$ is silting-discrete.

\item The poset $[P,P[\ell]]$ is finite for any $P\in \silt\T$ and any $\ell\in \mathbb{N}$.

\item There exists an object $P\in \silt\T$ such that $[P,P[\ell]]$ is finite for any $\ell\in \mathbb{N}$.

\item The poset $[P,P[1]]$ is finite for any $P\in \silt\T$.

\item The algebra $\End_{\T}(P)$ is $\tau$-tilting finite for any $P\in \silt\T$.
\end{enumerate}
\end{proposition}

Let $\T$ be a triangulated category with silting objects. A pre-silting object $P$ in $\T$ is called \emph{partial silting} if there exists an object $Q\in \T$ such that $P\oplus Q$ is a silting object. In general, the set of partial silting objects is much smaller than the set of pre-silting objects. Even in the case where $\T=\per(A)$,
the perfect derived category over a finite-dimensional algebra $A$, there are many examples of pre-silting objects that are not partial silting; see, for example, \cite{JSW23,LZ23}. However, there is a generalization of the Bongartz-type complement theorem for silting-discrete triangulated categories.

\begin{Prop}[{\cite[Theorem 2.15]{AM17}}]\label{Prop:silt}
Let $\T$ be a Hom-finite Krull-Schmidt triangulated category with silting objects. If $\T$ is silting-discrete, then any pre-silting object in $\T$ is partial silting. 
\end{Prop}

\subsection{Graded gentle algebras and their surface models}

We briefly recall some results on gentle algebras and marked surfaces. We refer to \cite{HKK17,OPS18,APS19,O19, LP20} for more details and examples.

We begin by recalling the definition of a graded gentle algebra.

\begin{Def}\label{definition:gentle algebras}
A \emph{graded gentle algebra} $A$ is the bound quiver algebra $\Bbbk Q/\langle I\rangle$, where $Q=(Q_0,Q_1, s, t)$ is a finite quiver and $\langle I\rangle$ is an ideal of $\Bbbk Q$, such that 
\begin{enumerate}
 \item Each vertex in $Q_0$ is the source of at most two arrows and the target of at most two arrows.

 \item For each arrow $\za$ in $Q_1$, there is at most one arrow $\zb$ such that $0 \neq \za\zb\in \langle I\rangle$; at most one arrow $\zg$ such that  $0 \neq \zg\za\in \langle I\rangle$; at most one arrow $\zb'$ such that $\za\zb'\notin \langle I\rangle$; at most one arrow $\zg'$ such that $\zg'\za\notin \langle I\rangle$.

 \item $\langle I\rangle$ is generated by paths of length two.

 \item The grading of $A$ is induced by associating to  each arrow $\alpha$  an integer $| \alpha |$.

\end{enumerate}
\end{Def}
For brevity, we call a graded gentle algebra with trivial grading a \emph{gentle algebra}.
Throughout, we will view a graded gentle algebra as a differential graded algebra with a trivial differential. Furthermore, in this paper, we only consider graded gentle algebras that are proper and homologically smooth, or equivalently, there are no cycles in $Q$ without relations, nor any cycles in $Q$ with full relations.

We now recall how to construct a graded gentle algebra from a graded marked surface.

\begin{Def}\label{definition: marked surface}
A pair~$(S,M)$ is called a \emph{marked surface} if
  \begin{enumerate}
 \item $S$ is an oriented surface with boundary with
  connected components $\partial S=\sqcup_{i=1}^{b}\partial_i $;
 \item $M = M^{\bpoint} \cup M^{\rpoint}$ is a finite set of \emph{marked points} on $S$.
       The elements of~$M^{\bpoint}$ and~$M^{\rpoint}$ are marked points on $\partial S$, which will be respectively represented by symbols~$\bpoint$ and~$\rpoint$. Each connected component $\partial_i $ is required to contain at least one marked point of each colour, where in general the points~$\bpoint$ and~$\rpoint$ are alternating on $\partial_i $.
  \end{enumerate}
\end{Def}

\begin{Def}\label{definition:arcs}
Let $(S,M)$ be a marked surface.
An \emph{$\bpoint$-arc} is a non-contractible curve, with endpoints in~$M^{\bpoint}$, while an \emph{$\rpoint$-arc} is a non-contractible curve, with endpoints in $M^{\rpoint}$.
\end{Def}

On the surface, all curves are considered up to homotopy, and all intersections of curves are required to be transversal. For brevity, in the following, we will denote a marked surface $(S,M)$ by $S$.

\begin{Def}\label{definition: dissection}
Let $S$ be a marked surface. 
\begin{enumerate}
 \item An \emph{admissible $\bpoint$-dissection} $\zD$ of $S$ is a collection of pairwise non-intersecting (in the interior of $S$) and pairwise distinct $\bpoint$-arcs  on the surface such that the arcs in $\zD$ cut the surface into polygons, each of which contains exactly one $\rpoint$-point.
 \item Let $q$ be a common endpoint of two arcs $\ell_i$ and $\ell_j$ in an admissible $\bpoint$-dissection $\zD$, an \emph{oriented intersection} from $\ell_i$ to $\ell_j$ at $q$ is an clockwise  angle locally from $\ell_i$ to $\ell_j$ based at $q$ such that the angle is in the interior of the surface. We call an oriented intersection a \emph{minimal oriented intersection} if it is not a composition of two oriented intersections of arcs in $\zD$.
 \item A \emph{graded admissible $\bpoint$-dissection} is a pair $(\zD,G)$, where $\zD$ is an admissible $\bpoint$-dissection and $G$ is a set of gradings, which is given by a set of  integers $g(\za)\in \mathbb{Z}$, one for every minimal oriented intersection $\za$.
 \item We call $(S,\zD, G)$ a \emph{graded marked surface}.
  \end{enumerate}
We similarly define \emph{admissible $\rpoint$-dissections} and \emph{graded admissible $\rpoint$-dissections}. \end{Def}

We introduce the following proposition-definition, whose proof is straightforward.

\begin{Prop-Def}\label{definition: dual graded dissection}
Let $(S,\zD,G)$ be a graded marked surface.
\begin{enumerate}
 \item
There is a unique (up to homotopy) admissible $\rpoint$-dissection $\zD^*$ on $S$ such that
 each arc $\ell^*$ in $\zD^*$ intersects exactly one arc $\ell$ of $\zD$, which we will call the \emph{dual $\rpoint$-dissection}
 of $\zD$.
  \item
  For any minimal oriented intersection $\za$ of $\zD$ from $\ell_i$ to $\ell_j$, there is a unique minimal oriented intersection $\za^{op}$ of $\zD^*$ from $\ell_j^*$ to $\ell_i^*$, see Figure \ref{figure:graded dual algebra}. Conversely, any minimal oriented intersection of $\zD^*$ arises from this way. We call $\za^{op}$ the \emph{dual minimal oriented intersection} of $\za$.

\item
The  \emph{ dual graded  $\rpoint$-dissection} $(\zD^*,G^*)$   of $(\zD,G)$ is the dual $\rpoint$-dissection $\zD^*$ of $\zD$ with dual of grading $G^*$ defined by setting $g(\za^{op}) = 1-g(\za)$, for any minimal oriented intersection $\za$ and its dual $\za^{op}$.

\end{enumerate}
\noindent We similarly define the dual graded $\bpoint$-dissection of an admissible graded $\rpoint$-dissection.
\end{Prop-Def}

Note that the dual graded dissection is unique (up to homotopy) and it is an involution, that is $(\zD^{**},G^{**})=(\zD,G)$.

\begin{figure}
\[\scalebox{1}{
\begin{tikzpicture}[>=stealth,scale=1]

\draw[red,thick,fill=red] (2,4) circle (0.1);

\draw[blue,thick,fill=blue] (2,0) circle (0.1);

\draw [thick,blue] (-1.3,3.8)--(2,0);
\draw [thick,blue](5.3,3.8)--(2,0);

\draw [thick,red] (-1.3,.2)--(2,4);
\draw [thick,red](5.3,.2)--(2,4);

\draw [thick,bend right,<-] (2.3,.35)to(1.7,.35);
\node at (2,.8) {$\za$};

\draw [thick,bend left,->] (2.3,3.65)to(1.7,3.65);
\node at (2,3.2) {$\za^{op}$};

\node at (1,1.7) {$\ell_j$};
\node at (3,1.7) {$\ell_i$};
\node at (1,2.3) {$\ell^*_j$};
\node at (3,2.3) {$\ell^*_i$};
\node at (2,-.4) {$q$};
\node at (2,4.4) {$q^*$};

\end{tikzpicture}
}\]
\begin{center}
\caption{Dual of minimal oriented intersections $\za$ and $\za^{op}$. We set $g(\za^{op}) = 1-g(\za)$ for the gradings.}\label{figure:graded dual algebra}
\end{center}
\end{figure}
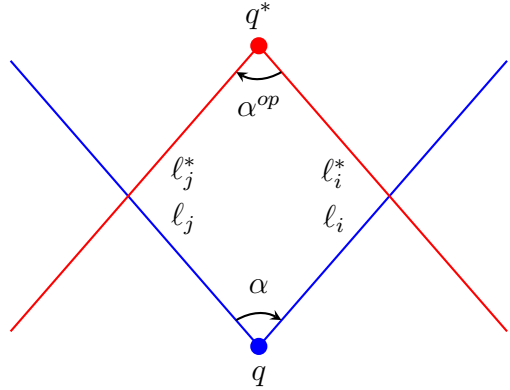

For ease of notation, we will often omit the grading $G$ of $(\zD, G)$ and simply write $\zD$ for a graded surface dissection.

We now recall how to associate a graded gentle algebra to a graded admissible dissection. For that,
let $(\zD, G)$ be a graded admissible $\bpoint$-dissection on a marked surface $S$. Define the graded algebra $A(\zD)=\Bbbk Q(\zD)/I(\zD)$ by setting:

\begin{enumerate}
  \item The vertices of $Q(\zD)$ are given by the arcs in $\zD$.
  \item Each minimal oriented intersection $\za$ from $\ell_i$ to $\ell_j$ gives rise to an arrow from $\ell_i$ to $\ell_j$, which is still denoted by $\za$. Furthermore, $|\alpha| = g(\za)$.
  \item  The ideal $I(\zD)$ is generated by the following relations: whenever $\ell_i$ and $\ell_j$ intersect at a marked point which gives rise to an arrow $\za:\ell_i\rightarrow\ell_j$, and the other end of $\ell_j$ intersects $\ell_k$ at a marked point which gives rise to an arrow $\zb:\ell_j\rightarrow\ell_k$, then the composition $\za\zb$ is a relation.
\end{enumerate}

Then $A(\zD)$ is a graded gentle algebra. In particular, this establishes a bijection between the set of homeomorphism classes of marked surfaces with graded admissible $\bpoint$-dissections and the set of isomorphism classes of graded gentle algebras \cite{BC21, OPS18,PPP18}. Similarly, we construct a graded gentle algebra $A(\zD^*)=\Bbbk Q(\zD^*)/I(\zD^*)$ from a graded admissible $\rpoint$-dissection $\zD^*$.

In \cite{HKK17}, a graded marked surface is a pair $(S, \eta)$, where $\eta$ is a line field on $S$ with the marked points as the stops. The line field induces a grading on each arc in an admissible dissection $\zD$ (and $\zD^*$), as well as a grading on each intersection of the arcs in $\zD$ (and $\zD^*$).
These two ways to define a grading on a surface are, in fact, equivalent. Indeed, it is shown in \cite{LP20} that a triple $(S, \zD, G)$ induces a line field $\eta$ on $S$ and thus gives rise to a graded marked surface $(S, \eta)$, where $S$ is unique up to diffeomorphism and the line field $\eta$ is unique up to homotopy. In the following, we will freely use $(S,\Delta,G)$ and $(S,\eta,\Delta)$, both denoting a graded marked surface with a dissection.

The winding number $\omega_\eta(\ell)$ of a closed immersed curve $\ell$ with respect to $\eta$ quantifies the ``twisting" of the curve relative to $\eta$, capturing a key topological relationship between the curve and the line field. For a formal definition of this invariant, we refer to \cite[Definition 1.3]{LP20}, where it is defined via the pairing of homology and cohomology classes associated with the line field and the tangent lift of the curve. Notably, the authors also introduce a combinatorial winding number, which is equivalent to the topological winding number defined above. This combinatorial version is detailed in the context of combinatorial boundary components, see \cite[Section 3.1]{LP20}, providing a more accessible computational framework. 

For each boundary component $\partial_i$, denote by $\omega_\eta(\partial_i)$ the \emph{winding number} $\partial_i$ with respect to the line field $\eta$. The following Poincar\'e-Hopf index formula will be used frequently in the paper:
\begin{equation}\label{eq:PH}
\sum_{1=1}^{b}\omega_\eta(\partial_i)=4-4g-2b,
\end{equation}
where $g$ is the genus of the $S$ and $b$ is the number of the boundary components of $S$.

\subsection{Skew-gentle algebras and their module categories}\label{section:skew-gentle}
We recall here the definition of skew-gentle algebras, which were introduced in the 1990's in \cite{GeissdelaPena}. This class of algebras had also been studied under the name clannish algebras in \cite{CB-skew-gentle, CB-skew-gentle-2, Deng00} and \cite{Geiss-clans}.

\begin{Def}\label{definition:skew-gentle algebras}
An algebra $A=\Bbbk Q/\langle I\rangle$ is said to be a \emph{skew-gentle algebra} if $Q=(Q_0,Q_1,s,t)$ is a finite quiver such that
\begin{itemize}
\item $Q_1=Q_1'\cup \mathsf{S}$ and $\mathsf{S}:=\{\epsilon_i \mid s(\epsilon_i)=t(\epsilon_i)=i\in Q_0\}$;

\item $I=I' \cup \{\epsilon_i^2-\epsilon_i\mid \epsilon_i\in \mathsf{S}\}$, where $I'$ is a set of paths containing no elements in $\mathsf{S}$;

\item $A':=\Bbbk Q'/\langle I'\rangle$ is a gentle algebra, where $Q':=(Q_0, Q_1')$;

\item if there exists an $\epsilon_i\in \mathsf{S}$, then it is the unique loop in $\mathsf{S}$ on vertex $i\in Q_0$ and there is at most one arrow in $Q_1'$ ending (resp., starting) at $i$; moreover, we have $\alpha\beta\in I'$ if both arrows $\alpha$ and $\beta$ at the source (and target) of $\epsilon_i \in \mathsf{S}$
\begin{center}
$\xymatrix@C=1cm{\ar[r]^-\alpha&i\ar@(ul,ur)^-{\epsilon_i}\ar[r]^-\beta &}$ 
\end{center}
exist. 
\end{itemize}
\end{Def}

Note that there is a special case of Definition~\ref{definition:skew-gentle algebras}, namely when in the last item $\alpha=\beta$  is a loop. Then we obtain the local algebra $\Bbbk \langle \epsilon, \alpha \mid \epsilon^2=\epsilon, \alpha^2=0 \rangle$ studied in \cite{CB-skew-gentle}.

Elements in $\mathsf{S}$ are called \emph{special loops}, and a vertex $i\in Q_0$ is said to be \emph{special} if there is a special loop $\epsilon_i$ incident to $i$. A gentle algebra can be regarded as a skew-gentle algebra without special loops. We mention that in the above presentation $A=\Bbbk Q/\langle I \rangle$ of a skew-gentle algebra, the ideal $\langle I \rangle$ is not admissible except when $A$ is a gentle algebra.

In the following, we recall the constructions of the indecomposable modules over a skew-gentle algebra in terms of strings and bands as initiated in \cite{CB-skew-gentle} (for gentle algebras see \cite{AS-tilting-cotilting-equ, BR87, WW85}) and further developed in \cite{CB-skew-gentle-2}, see also  \cite{B91,Deng00,Hansper-thesis}. Here we mostly  follow the presentation in  \cite{Hansper-thesis}, with some reformulations and adjustments  to meet our purposes.

Let $A=\Bbbk Q/\langle I\rangle$ be a skew-gentle algebra. For each arrow $\alpha\in Q'_1$, we denote by $\alpha^{-1}$ its formal inverse with $s(\alpha^{-1})=t(\alpha)$ and $t(\alpha^{-1})=s(\alpha)$. Furthermore, we set $(\alpha^{-1})^{-1}=\alpha$ such that $(-)^{-1}$ gives an involution on the set of arrows in $Q'_1$ and their formal inverses. For each special loop $\epsilon \in\mathsf{S}$, we set $\epsilon^{-1}:=\epsilon$; in particular, we usually write $\epsilon^\ast:=\epsilon=\epsilon^{-1}$ to indicate $\epsilon \in\mathsf{S}$. We then obtain a set of letters:
$$
\Gamma(A):=\{\alpha, \alpha^{-1}, \epsilon^\ast \mid \alpha\in Q_1', \epsilon\in \mathsf{S}\}.
$$
A finite sequence $w:=w_1w_2\cdots w_n$ with $w_i\in \Gamma(A)$ is called a \emph{word  of length $n\ge 1$} if it satisfies 
\begin{enumerate}
\item $t(w_i)=s(w_{i+1})$, for any two consecutive letters $w_i, w_{i+1}$;
\item $w_i^{-1}\neq w_{i+1}$, for any two consecutive letters $w_i, w_{i+1}$;
\item $\epsilon^\ast\epsilon^\ast$ does not occur as a subsequence of $w$, for any $\epsilon\in \mathsf{S}$;
\item any subsequence $u=u_1u_2\cdots u_k$ of $w$ is not in $I'$, and neither is its inverse $u^{-1}=u_k^{-1}\cdots u_2^{-1}u_1^{-1}$.
\end{enumerate}
We call a subsequence of a word $w$ a \emph{subword} of $w$. For each vertex $i\in Q_0$, there is a \emph{word of length $0$}, denoted by $1_i$, that starts and ends at $i$. We denote by $s(w):=s(w_1)$ and $t(w):=t(w_n)$ the source and target of a word $w$, respectively. A word $w$ is called an \emph{ordinary string} if there is no $\epsilon^\ast$ appearing in $w$, and it is called \emph{coadmissible} if neither $\epsilon^\ast w$ nor $w\epsilon^\ast$ is a word, for any special loop $\epsilon\in \mathsf{S}$.

\begin{remark}
Let $w=w_1w_2\cdots w_n$ be a word. For any $\epsilon\in \mathsf{S}$, $\epsilon^\ast w$ (resp., $w\epsilon^\ast$) is not a word if and only if either $w_1$ (resp., $w_n$) is a special loop, or $w_1$ (resp., $w_n$) is not a special loop and there is no special loop incident to $s(w_1)$ (resp., $t(w_n)$). 
\end{remark}

\begin{definition}
Let $w$ be a word. 
\begin{enumerate}
    \item It is called an asymmetric string if $w$ is coadmissible and $w\neq w^{-1}$.
    \item It is called a symmetric string if $w$ is coadmissible and $w=u\epsilon^\ast u^{-1}$, where $u\neq u^{-1}$ and $u\neq (v\epsilon^\ast v^{-1}\eta^\ast)^k v$, for any subword $v$ with $k\ge 1$, $s(v)=t(\eta)$, $\eta\in \mathsf{S}$.
\end{enumerate}
We sometimes refer to asymmetric and symmetric strings collectively as strings.
\end{definition}

\begin{remark}
Suppose $A$ is a gentle algebra. Then, each word $w=w_1w_2\cdots w_n$ is an ordinary string; this is called \emph{string} in most literature. Since if $w=w^{-1}$ and $n=2k$ for some $k\in \N$, we obtain $w_k=w_{k+1}^{-1}$ which contradicts to the definition of word; if $w=w^{-1}$ and $n=2k+1$ for some $k\in \N$, we obtain $w_{k+1}=w_{k+1}^{-1}=\epsilon^\ast$ which contradicts to the definition of ordinary string. We conclude that in this case, each ordinary string is an asymmetric string. 
\end{remark}

\begin{example}\label{ex::string}
Let $A=\Bbbk Q/I$ be the skew-gentle algebra given by 
$$
Q: \quad \vcenter
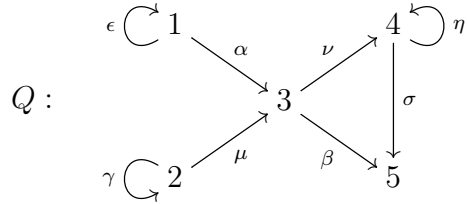
{\xymatrix@C=1cm@R=0.5cm{
1\ar@(ld,lu)^{\epsilon}\ar[dr]^{\alpha}&&4\ar@(rd,ru)_{\eta}\ar[dd]^-{\sigma}\\
&3\ar[ur]^-{\nu}\ar[dr]_-{\beta}&\\
2\ar@(lu,ld)_{\gamma}\ar[ur]_{\mu}&&5
}}
$$
and $I=\langle \alpha\beta, \mu\nu, \nu\sigma, \gamma^2,\epsilon^2-\epsilon, \eta^2-\eta\rangle$. Then, $\gamma\mu\beta\sigma^{-1}$ is not coadmissible while $\gamma\mu\beta\sigma^{-1}\eta^\ast$ is coadmissible. We find that $\gamma\mu\beta\sigma^{-1}\eta^\ast$ is an asymmetric string and $\gamma\mu\beta\sigma^{-1}\eta^\ast\sigma \beta^{-1}\mu^{-1}\gamma^{-1}$ is a symmetric string. Besides, $\epsilon^\ast \alpha\nu\eta^\ast\sigma \beta^{-1}$ is also an asymmetric string.
\end{example}

We call an infinite sequence $w_{\mathbb{Z}}:=\cdots w_{-1}w_{0} \mid w_1w_2w_3\cdots$ with $w_i\in \Gamma(A)$ a \emph{$\mathbb{Z}$-word} if each finite subsequence of $w_{\mathbb{Z}}$ is a word. The symbol $|$ indicates the position of $0$-th letter. A finite subsequence of a $\mathbb{Z}$-word $w_{\mathbb{Z}}$ is called a subword of $w_{\mathbb{Z}}$.

Given a $\mathbb{Z}$-word $w_{\mathbb{Z}}$, we call $w_{\mathbb{Z}}[k]$ (or $w_{\mathbb{Z}}[-k]$) the shift of $w_{\mathbb{Z}}$ obtained by moving the symbol $|$ to the right (or left) by $k$ steps for $k\in \mathbb{Z}_{>0}$. If there is a minimal $p\in \mathbb{Z}_{>0}$ with the property $w_{\mathbb{Z}}=w_{\mathbb{Z}}[p]$, we call $p$ the period of $w_{\mathbb{Z}}$, and rewrite $w_{\mathbb{Z}}=\cdots w_1w_2\cdots w_p\mid w_1w_2\cdots w_p\cdots$. Moreover, the subword $\widehat{w}:=w_1w_2\cdots w_p$ is called the \emph{periodic part of $w_{\mathbb{Z}}$}. It is easy to check that the shift of $w_{\mathbb{Z}}$ is additive, that is, $w_{\mathbb{Z}}[m+n]=(w_{\mathbb{Z}}[m])[n]$. We also have $(w_{\mathbb{Z}}[k])^{-1}=w_{\mathbb{Z}}^{-1}[-k]$, for any $k\in \mathbb{Z}$.

\begin{definition}
Let $w_{\mathbb{Z}}$ be a $\mathbb{Z}$-word.
\begin{enumerate}
    \item It is said to be an asymmetric band if $w_{\mathbb{Z}}=w_{\mathbb{Z}}[p]$ for some $p>0$, and $w_{\mathbb{Z}}\neq w_{\mathbb{Z}}^{-1}[k]$ for any $k\in \mathbb{Z}$.
    \item It is said to be a symmetric band if $w_{\mathbb{Z}}=w_{\mathbb{Z}}[p]$ for some $p>0$, and $w_{\mathbb{Z}}= w_{\mathbb{Z}}^{-1}[k]$ for some $k\in \mathbb{Z}$.
\end{enumerate}
We sometimes refer to asymmetric and symmetric bands collectively as bands.
\end{definition}

\begin{remark}
Let $A$ be a gentle algebra. In the literature, a \emph{band} is an asymmetric string $w$ of length $\ge 1$ such that $s(w)=t(w)$, all powers of $w$ are asymmetric strings, and $w$ is not itself a power of a strictly smaller asymmetric string. We then conclude that in this case, a band $w$ is exactly the periodic part of some asymmetric band $w_\mathbb{Z}$.
\end{remark}

\begin{example}
Following Example \ref{ex::string}, $\cdots\gamma\mu\alpha^{-1}\epsilon^\ast \alpha \mu^{-1}\mid \gamma\mu\alpha^{-1}\epsilon^\ast \alpha \mu^{-1}\cdots$ is an asymmetric band, and $\cdots\epsilon^\ast \alpha\nu\eta^\ast\nu^{-1} \alpha^{-1}\mid \epsilon^\ast \alpha\nu\eta^\ast\nu^{-1} \alpha^{-1}\cdots$ is a symmetric band.
\end{example}

\begin{proposition}[{\cite[Lemma 3.36]{Hansper-thesis}}]\label{prop::H-symmetric-band}
Let $w_{\mathbb{Z}}$ be a symmetric band. Then, the periodic part $\widehat{w}$ of $w_{\mathbb{Z}}$ is of the form $\epsilon^\ast u \eta^\ast u^{-1}$, for a suitable non-trivial word $u$ and $\epsilon,\eta\in \mathsf{S}$.
\end{proposition}

Let $w$ be a word or $\mathbb{Z}$-word. One may assign a direction to each $\epsilon^\ast$ appearing in $w$, and then obtain a directed version $ \overrightarrow{w}$ of $w$. We then put a $\Bbbk$-vector space on each vertex along the path read from $\overrightarrow{w}$ with an action of $A$, such that it gives an $A$-module $M(\overrightarrow{w})$.
We refer to the literature, for example, \cite{CB-skew-gentle, CB-skew-gentle-2,B91,Deng00,Hansper-thesis} for the specific construction of $M(\overrightarrow{w})$ and also for the following result.

\begin{theorem}\label{theo::classification}
Let $A$ be a skew-gentle algebra. Then, the set of modules $M(\overrightarrow{w})$ with $w$ running through all possible strings and bands, gives a complete set of pairwise non-isomorphic indecomposable $A$-modules.
\end{theorem}

\section{$\tau$-tilting finiteness of  skew-gentle algebras}

It is known from \cite{P19} that a gentle algebra is $\tau$-tilting finite if and only if it is representation-finite. In this section, we are aiming to establish a similar result for skew-gentle algebras.
We first introduce the definition of \emph{minimal} bands.

\begin{definition}
\label{prop::minimal-asymmetric-band}
Let $w_{\mathbb{Z}}$ be an asymmetric band for a skew-gentle algebra $A$. We say that $w_{\mathbb{Z}}$ is \emph{minimal} if the periodic part $\widehat{w}$ of $w_{\mathbb{Z}}$ is one of the following form:
\begin{enumerate}
\item $\widehat{w}$ does not go through the same vertex twice except for the source and target of $\widehat{w}$, i.e., $\widehat{w}$ is read from a quiver of type $\widetilde{\mathbb{A}}$.

\item $\widehat{w}$ has the form $\widehat{w}=uzu'z^{-1}$, where
\begin{itemize}
    \item $u$ and $u'$ are ordinary strings with $s(u)=t(u)$ and $s(u')=t(u')$, $u^2$ and $(u')^2$ are not ordinary strings;
    \item none of $u, u'$ and $z$ go through the same vertex twice except for the sources of $u$ and $u'$;
    \item $z$ is a possibly trivial ordinary string $1_i$ for some vertex $i\in Q_0$, and $i$ is the only vertex that $u, u', z$ may have in common;
\end{itemize}
i.e., $\widehat{w}$ is read from a quiver of the form 
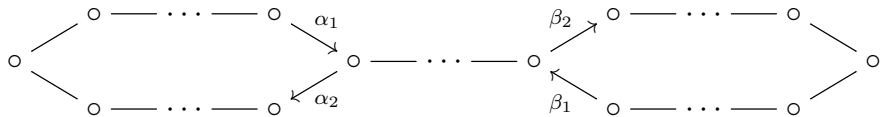
\begin{center}
$\vcenter{\xymatrix@C=0.6cm@R=0.2cm{
&\circ\ar@{-}[r]&\cdots\ar@{-}[r]&\circ\ar[dr]^{\alpha_1}&&&&\circ\ar@{-}[r]&\cdots\ar@{-}[r]&\circ\ar@{-}[dr]&\\
\circ\ar@{-}[ur]\ar@{-}[dr]&&&&\circ\ar@{-}[r]\ar[dl]^{\alpha_2}&\cdots\ar@{-}[r] &\circ\ar[ur]^-{\beta_2}&&&&\circ\\
&\circ\ar@{-}[r]&\cdots\ar@{-}[r]&\circ&&&&\circ\ar[ul]^-{\beta_1}\ar@{-}[r]&\cdots\ar@{-}[r]&\circ\ar@{-}[ur]&
}}$
\end{center}
with $\alpha_1\alpha_2=\beta_1\beta_2=0$.

\item $\widehat{w}$ has the form $\widehat{w}=\epsilon^\ast z u z^{-1}$ or $uz\epsilon^\ast z^{-1}$ for some $\epsilon\in \mathsf{S}$, where
\begin{itemize}
    \item $u$ is an ordinary string with $s(u)=t(u)$, and $u^2$ is not an ordinary string;
    \item none of $u$ and $z$ go through the same vertex twice except for the source of $u$;
    \item $z$ is a possibly trivial ordinary string $1_{s(\epsilon)}$, and $s(\epsilon)$ is the only vertex that $u, \epsilon, z$ may have in common;
\end{itemize}
i.e., $\widehat{w}$ is read from a quiver of the form 
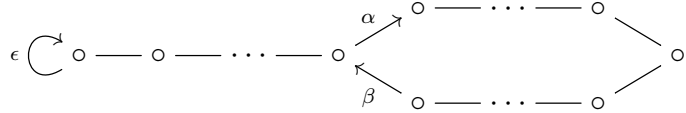
\begin{center}
$\vcenter{\xymatrix@C=0.6cm@R=0.2cm{
&&&&\circ\ar@{-}[r]&\cdots\ar@{-}[r]&\circ\ar@{-}[dr]&\\
\circ\ar@(ld,lu)^{\epsilon}\ar@{-}[r]&\circ\ar@{-}[r]&\cdots\ar@{-}[r] &\circ\ar[ur]^-{\alpha}&&&&\circ\\
&&&&\circ\ar[ul]^-{\beta}\ar@{-}[r]&\cdots\ar@{-}[r]&\circ\ar@{-}[ur]&
}}$
\end{center}
with $\beta\alpha=0$ and $\epsilon\in \mathsf{S}$.
\end{enumerate}
\end{definition}

\begin{definition}\label{pro::minimal-symmetric-band}
Let $A$ be a skew-gentle algebra. A symmetric band $w_{\mathbb{Z}}$ is said to be \emph{minimal} if the periodic part $\widehat{w}$ of $w_{\mathbb{Z}}$ has the form $\epsilon^\ast z \eta^\ast z^{-1}$, where $z$ cannot be a trivial ordinary string, $z$ does not go through the same vertex twice including the source and target of $z$, and $\epsilon,\eta\in \mathsf{S}$. In other words, $w$ is read from a quiver of the form
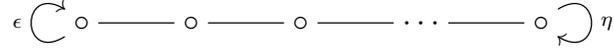
\begin{center}
$\xymatrix@C=1cm{\circ\ar@(ld,lu)^{\epsilon}\ar@{-}[r]&\circ\ar@{-}[r]&\circ\ar@{-}[r]&\cdots\ar@{-}[r] &\circ\ar@(ru,rd)^{\eta}}$
\end{center}
with $\epsilon,\eta\in \mathsf{S}$.
\end{definition}

We remark that, for a gentle algebra, the minimal bands occur only in types (1) and (2) as described in Definition \ref{prop::minimal-asymmetric-band}. It was shown in \cite{P19} that any gentle algebra admitting a band always admits a minimal band. The following result extends this property to  skew-gentle algebras.

\begin{proposition}\label{prop::minimal-exist}
Let $A$ be a skew-gentle algebra. If $A$ admits a band, then $A$ admits at least one minimal band.
\end{proposition}
\begin{proof}
Let $w$ be a band and $A':=A/\langle \mathsf{S}\rangle$ the gentle quotient of $A$ modulo all special loops. If $\widehat{w}$ contains no $\epsilon^\ast$ for any $\epsilon\in \mathsf{S}$, then $\widehat{w}$ is a band over $A'$. It is shown in \cite{P19} that $\widehat{w}$ gives rise to at least one minimal asymmetric band, which is of either case (1) or case (2) in Definition \ref{prop::minimal-asymmetric-band}. We suppose in the following that $\widehat{w}$ contains at least one special loop $\epsilon^\ast$. 

We rewrite $w$ as $\cdots \epsilon^\ast u_2u_3\cdots u_n\epsilon^\ast u_2u_3\cdots u_n\cdots$, and set $u:=\epsilon^\ast u_2u_3\cdots u_n$, where $n\ge 2$. There is a sequence of vertices
\begin{equation}\label{sequence}
s(\epsilon^\ast),t(\epsilon^\ast), s(u_2), t(u_2), \cdots, s(u_i), t(u_i), \cdots, s(u_n), t(u_n).
\end{equation}
We set $i:=s(\epsilon^\ast)=t(\epsilon^\ast)=s(u_2)=t(u_n)$. If $i$ is the unique vertex appearing repeated in \eqref{sequence}, then $u$ is the periodic part of a minimal band as defined in Definition \ref{prop::minimal-asymmetric-band} (3), in which $z$ is a trivial ordinary string. Otherwise, suppose $j:=s(u_x)=t(u_y)$ with $3\le x\le y \le n-1$ to be another vertex appearing repeated in $u$, such that both $x$ and $y$ are minimal. Then we have
\begin{itemize}
\item If $u_x$ is an arrow or an ordinary loop, then $\epsilon^\ast u_2u_3\cdots u_{x-1}(u_x\cdots u_y) u_{x-1}^{-1} \cdots u_3^{-1}u_2^{-1}$ is the periodic part of a minimal band as defined in case (3) of Definition \ref{prop::minimal-asymmetric-band}. 
\item If $u_x=\eta^\ast=u_y$ for some special loop $\eta\in \mathsf{S}$, then $\epsilon^\ast u_2u_3\cdots u_{x-1}\eta^\ast u_{x-1}^{-1} \cdots u_3^{-1}u_2^{-1}$ is the periodic part of a minimal band as defined in Definition \ref{pro::minimal-symmetric-band}. 
\end{itemize}
We conclude that $u$ induces a minimal band of $A$.
\end{proof}
 
We recall an admissible presentation for a skew-gentle algebra $A=\Bbbk Q/\langle I\rangle$ as follows. As in the proof of the previous proposition, we note that  $A$ contains a gentle algebra $A'=\Bbbk Q'/I'$ as a quotient, obtained by deleting all special loops in $\mathsf{S}$. Set $Q^s=(Q_0^s, Q_1^s)$ with 
\begin{itemize}
\item $Q_0^s:=\bigcup_{i\in Q_0}Q_0^s(i)$, where  
\begin{center}
$\begin{aligned}
Q_0^s(i):= \left\{
\begin{array}{ll}
\{i^+, i^-\} & \text{ if } \exists\ \epsilon_i\in \mathsf{S}\\
\{ i\} & \text{ otherwise }
\end{array}\right.
\end{aligned}$, 
\end{center}

\item $Q_1^s:=\{(i,\alpha, j)\mid \alpha\in Q_1', i\in Q_0^s(s(\alpha)), j\in Q_0^s(t(\alpha))\}$.
\end{itemize}
We define an admissible ideal $I^s$ of $\Bbbk Q^s$ by
\begin{center}
$I^s:=\left < \sum\limits_{j\in Q_0^s(s(\beta))} \lambda_j(i, \alpha, j)(j,\beta, k)\mid \alpha\beta\in I', i\in Q_0^s(s(\alpha)), k\in Q_0^s(t(\beta)) \right >$,
\end{center}
where $\lambda_j=-1$ if $j=l^-$ for some $l\in Q_0$, and $\lambda_j=1$ otherwise. Then, the algebras $A$ and  $A^s:=\Bbbk Q^s/I^s$  are isomorphic. 

We now show that the existence of a minimal band implies that the algebra is  $\tau$-tilting infinite. We begin by recalling that  each path algebra of extended Dynkin type is $\tau$-tilting infinite, see, for example, \cite[Theorem 3.1]{Ada-rad-square-0}.
\begin{lemma}\label{lem::S_1}
Let $\mathbf{S}_1:=\Bbbk Q/I$ be the skew-gentle algebra given by
$$
Q: \quad \vcenter
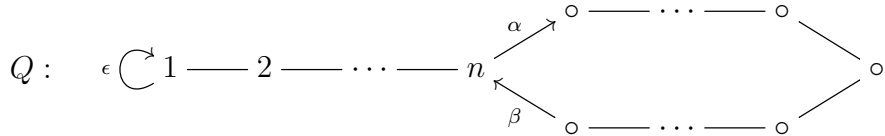
{\xymatrix@C=0.8cm@R=0.3cm{
&&&&\circ\ar@{-}[r]&\cdots\ar@{-}[r]&\circ\ar@{-}[dr]&\\
1\ar@(ld,lu)^{\epsilon}\ar@{-}[r]&2\ar@{-}[r]&\cdots\ar@{-}[r] &n\ar[ur]^-{\alpha}&&&&\circ\\
&&&&\circ\ar[ul]^-{\beta}\ar@{-}[r]&\cdots\ar@{-}[r]&\circ\ar@{-}[ur]&
}}
$$
with $I=\langle \{\beta\alpha\}\cup\{ \epsilon^2-\epsilon\} \rangle$, that is $\mathsf{S}=\{\epsilon\}$. Then, $\mathbf{S}_1$ is $\tau$-tilting infinite.
\end{lemma}
\begin{proof}
If $n=1$, then $\mathbf{S}_1$ is isomorphic to the bound quiver algebra $\mathbf{S}^s_1$ given by 
$$
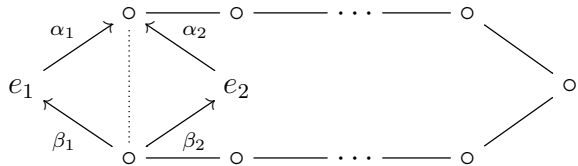
\vcenter{\xymatrix@C=0.9cm@R=0.5cm{
&\circ\ar@{-}[r]\ar@{.}[dd]&\circ\ar@{-}[r]&\cdots\ar@{-}[r]&\circ\ar@{-}[dr]&\\
e_1\ar[ur]^-{\alpha_1}&&e_2\ar[ul]_-{\alpha_2}&&&\circ\\
&\circ\ar[ul]^-{\beta_1}\ar[ur]_-{\beta_2}\ar@{-}[r]&\circ\ar@{-}[r]&\cdots\ar@{-}[r]&\circ\ar@{-}[ur]&
}}
$$
with $\beta_1\alpha_1=\beta_2\alpha_2$. Set $A:=e\mathbf{S}^s_1e$ with $e:=1-e_1-e_2$.  Then, $A$ is  $\tau$-tilting infinite since it is a path algebra of type $\tilde{\mathbb{A}}$. We conclude that $\mathbf{S}_1$ is $\tau$-tilting infinite.

Suppose $n\ge 2$. We find that $\mathbf{S}_1$ is isomorphic to the bound quiver algebra $\mathbf{S}^s_1$ given by 
$$
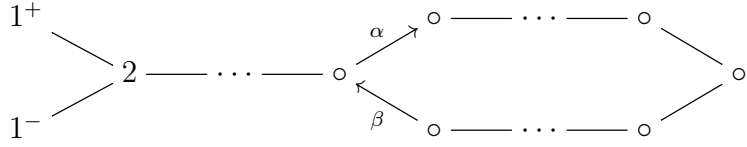
\vcenter{\xymatrix@C=0.8cm@R=0.2cm{
1^+\ar@{-}[dr]&&&&\circ\ar@{-}[r]&\cdots\ar@{-}[r]&\circ\ar@{-}[dr]&\\
&2\ar@{-}[r]&\cdots\ar@{-}[r] &\circ\ar[ur]^-{\alpha}&&&&\circ\\
1^-\ar@{-}[ur]&&&&\circ\ar[ul]^-{\beta}\ar@{-}[r]&\cdots\ar@{-}[r]&\circ\ar@{-}[ur]&
}}
$$
with $\beta\alpha=0$. Then, we define, $\lambda\in \Bbbk$,
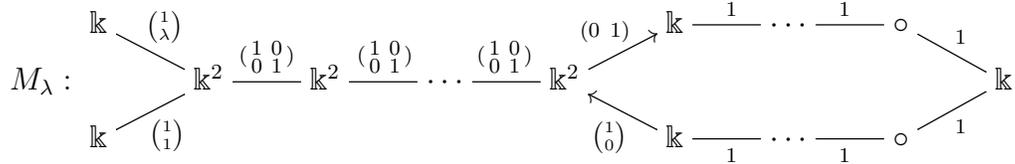
$$
M_\lambda :\vcenter{\xymatrix@C=0.9cm@R=0.2cm{
\Bbbk\ar@{-}[dr]^-{\binom{1}{\lambda}}&&&&&\Bbbk\ar@{-}[r]^1&\cdots\ar@{-}[r]^1&\circ\ar@{-}[dr]^1&\\
&\Bbbk^2\ar@{-}[r]^{(\begin{smallmatrix}
1 & 0 \\
0 & 1 \\
\end{smallmatrix})}&\Bbbk^2\ar@{-}[r]^{(\begin{smallmatrix}
1 & 0 \\
0 & 1 \\
\end{smallmatrix})}&\cdots\ar@{-}[r]^{(\begin{smallmatrix}
1 & 0 \\
0 & 1 \\
\end{smallmatrix})} &\Bbbk^2\ar[ur]^-{(0\ 1)}&&&&\Bbbk\\
\Bbbk\ar@{-}[ur]_{\binom{1}{1}}&&&&&\Bbbk\ar[ul]^-{\binom{1}{0}}\ar@{-}[r]_1&\cdots\ar@{-}[r]_1&\circ\ar@{-}[ur]_1&
}}
$$
which is a $\mathbf{S}^s_1$-module. One can check that 
\begin{itemize}
\item $M_\lambda\simeq M_{\lambda'}$ if and only if $\lambda=\lambda'\in \Bbbk$, 
\item $\mathrm{End}_{\mathbf{S}^s_1}(M_\lambda)=\{x\cdot 1\mid x\in \Bbbk\}\simeq \Bbbk$.
\end{itemize}
This implies that $M_\lambda$ is a brick and $(M_\lambda)_{\lambda\in \Bbbk}$ gives an infinite family of bricks. Using Proposition \ref{prop::brick-tau-rigid}, we conclude that $\mathbf{S}_1$ is $\tau$-tilting infinite.
\end{proof}

\begin{lemma}\label{lem::S_2}
Let $\mathbf{S}_2:=\Bbbk Q/I$ be the skew-gentle algebra given by
$$
Q: \quad \xymatrix@C=1cm{1\ar@(ld,lu)^{\epsilon_1}\ar@{-}[r]&2\ar@{-}[r]&3\ar@{-}[r]&\cdots\ar@{-}[r] &n\ar@(ru,rd)^{\epsilon_n}}\quad (n\ge 2)
$$
with $I=\langle \epsilon^2_1 - \epsilon_1, \epsilon^2_2 - \epsilon_2 \rangle$ that is $\mathsf{S}=\{\epsilon_1, \epsilon_n\}$. Then, $\mathbf{S}_2$ is $\tau$-tilting infinite.
\end{lemma}
\begin{proof}
It is obvious that $\mathbf{S}_2$ is isomorphic to a hereditary algebra of type $\widetilde{\mathbb{D}}_{n+1}$, that is, the path algebra with any choice of orientation given by 
$$
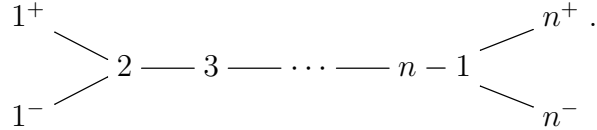
\xymatrix@C=0.7cm@R=0.1cm{
1^+\ar@{-}[dr]&&&&&n^+\ar@{-}[dl]\\
&2&3\ar@{-}[r]\ar@{-}[l]&\cdots&n-1\ar@{-}[l]&\\
1^-\ar@{-}[ur]&&&&&n^-\ar@{-}[ul]
}.
$$
It is $\tau$-tilting infinite as shown in \cite[Theorem 3.1]{Ada-rad-square-0}.
\end{proof}

Now we are ready to state our main result in this section.
\begin{theorem}\label{theo::main-result-skew-gentle}
Let $A$ be a skew-gentle algebra. Then, $A$ is $\tau$-tilting finite if and only if $A$ is representation-finite, or equivalently, $A$ admits no band.
\end{theorem}
\begin{proof}
($\Leftarrow$) If $A$ admits no band, then it is representation-finite  and hence $A$ is $\tau$-tilting finite.

($\Rightarrow$) Suppose $A$ admits a band $w_{\mathbb{Z}}$. According to Proposition \ref{prop::minimal-exist}, we may further assume that $w_{\mathbb{Z}}$ is a minimal band. If the periodic part $\widehat{w}$ of $w_{\mathbb{Z}}$ is an ordinary string, then the gentle quotient $A'$ of $A$ is $\tau$-tilting infinite, shown in \cite{P19}. 

Let $J$ be the ideal generated by all idemponents associated with vertices that $w_{\mathbb{Z}}$ does not go through. If $\widehat{w}$ is an asymmetric band in the case of Definition \ref{prop::minimal-asymmetric-band} (3), then $A/J$ is isomorphic to $\mathbf{S}_1$. If $\widehat{w}$ is a symmetric band, then $A/J$ is isomorphic to $\mathbf{S}_2$. Both are $\tau$-tilting infinite, according to Lemma \ref{lem::S_1} and Lemma \ref{lem::S_2}.
\end{proof}

\section{Silting-discreteness of graded gentle algebras}

In this section, we give a classification of the graded gentle algebra for which the perfect derived category is silting discrete  via their associated surface models. Our main result is the following.

\begin{theorem}\label{thm:main1}
    Let $A$ be a homologically smooth, proper graded gentle algebra and let $(S,\eta,\Delta)$ be the associated graded surface. Then $\per(A)$ is silting-discrete if and only if $S$ is of genus zero and $\omega_\eta(\ell)\neq 0$ for any simple closed curve $\ell$.
\end{theorem}

We divide the proof of Theorem \ref{thm:main1} into two parts: the necessary and sufficient conditions.

\subsection{The necessary part of Theorem \ref{thm:main1}}

Let $A = \Bbbk Q/\langle I\rangle$  be a homologically smooth and proper graded gentle algebra arising from a graded marked surface with dissection $(S,\eta,\Delta)$. 
It is clear that $\per(A)$ is Krull-Schmidt and Hom-finite. One useful property of silting-discreteness is that any pre-silting object is partial silting, see Proposition \ref{Prop:silt}. 
For graded gentle algebras, we have the following complete characterization of this property due to \cite{JSW23}. Recall that for a graded surface $(S,\eta)$ of genus one, we denote by $s$ a simple closed curve around the hole of $S$
and by $t$ a simple closed curve around the tube of $S$.

\begin{Thm}\cite{JSW23}\label{JSW}
Let $A$ be a graded gentle algebra arising from a graded surface with dissection $(S,\eta,\Delta)$.
The following two conditions are equivalent:
\begin{enumerate}
    \item Any pre-silting object in $\per(A)$ is partial silting;
    \item $S$ is of genus zero, or $S$ is of genus one such that 
\begin{equation}\label{equ:neq}
    \Tilde{A}(\eta)\neq \gcd(\omega_\eta(\partial_1)+2,\omega_\eta(\partial_2)+2,\cdots,\omega_\eta(\partial_b)+2)
\end{equation}
where $\Tilde{A}(\eta)=\gcd(\omega_\eta(s),\omega_\eta(t),\omega_\eta(\partial_1)+2,\omega_\eta(\partial_2)+2,\cdots,\omega_\eta(\partial_b)+2)$.
\end{enumerate} 
\end{Thm}

By Proposition \ref{Prop:equ-silting-discrete}, to show that $\per(A)$ is silting-discrete, it is equivalent to showing that for any silting object $P$ in $\per(A)$, the algebra $\End_{\per(A)}(P)$ is $\tau$-tilting finite.
For a finite-dimensional gentle algebra, it is shown in \cite{P19} that $\tau$-tilting finiteness is equivalent to representation-finiteness, or equivalently, to the non-existence of band modules.

Based on this, we first show the following lemma, which gives the first part of the necessary condition in Theorem \ref{thm:main1}. 

\begin{lemma}\label{lem:key1}
    If $\per(A)$ is silting-discrete, then the genus of $S$ is zero.
\end{lemma}

\begin{proof}
Since $\per(A)$ is silting-discrete, by Proposition \ref{Prop:silt} and Theorem \ref{JSW}, $S$ is either of genus zero or of genus one with condition \eqref{equ:neq} holding.
We now show that if $S$ is of genus one and condition \eqref{equ:neq} holds, then $\per(A)$ is not silting-discrete.

Let $m=\gcd(\omega_\eta(s),\omega_\eta(t))$. Note that \eqref{equ:neq} is equivalent to the condition that
\[\gcd(\omega_\eta(\partial_1)+2,\omega_\eta(\partial_2)+2,\cdots,\omega_\eta(\partial_b)+2)\nmid m.\] 
In particular, $m$ is nonzero. So we can assume that $m>0$. 

It is known that line fields on $S$ are classified up to homotopy by their winding numbers on the boundary components and the handle loops. Conversely, note that the space of line fields is an $\h^1(S,\mathbb{Z})$-affine space, so any assignment of integers to the collection of all boundary components and handle cycles that respects the Poincar\'e-Hopf index formula corresponds to a unique cohomology class in $\h^1(S,\mathbb{Z})$, and thus determines a line field up to homotopy (see details in \cite[Proposition 3.4 (5)]{APS19} and \cite{C72}). 

Thus we may always construct a new line field $\eta'$ on $S$ such that $\omega_{\eta'}(\partial_i)=\omega_{\eta}(\partial_i)$ for each $1\leq i\leq b$, and $\omega_{\eta'}(s)=\omega_{\eta'}(t)=-m$.  Note that $\Tilde{A}(\eta)=\Tilde{A}(\eta')$, thus
there exists an orientation-preserving homeomorphism $\phi: S\rightarrow S$ such that $\phi(\eta')$ is homotopic to $\eta$ by \cite[Corollary 1.10]{LP20}.

Now let $\Delta'$ be a graded admissible dissection on $(S,\eta')$ such that the associated graded gentle algebra $A'$ contains the following sub-quiver (for example, we can consider the `standard  dissection'  introduced in \cite[Section 4.1]{JSW23}):
\begin{equation}
\label{equ:An}
\begin{tikzpicture}[
    >=stealth,              
    node distance=1.5cm,    
    baseline=(current bounding box.center),
    font=\small,
    thick                   
]

    \node (1) {$1$};
    \node (2) [right=of 1] {$2$};
    
    \node[draw=gray, dashed, rounded corners=3pt, inner sep=4pt] (Q) [right=of 2] {$Q''$};

    \def\vshift{1em} 

    
    \draw[->] ([yshift=\vshift]1.east) -- node[above, font=\scriptsize, inner sep=1pt] {$\alpha$} ([yshift=\vshift]2.west);
    
    \draw[->] (2.west) -- node[above, font=\scriptsize, inner sep=1pt] {$\beta$} (1.east);
    
    \draw[->] ([yshift=-\vshift]1.east) -- node[above, font=\scriptsize, inner sep=1pt] {$\gamma$} ([yshift=-\vshift]2.west);

    \draw[->] (2) -- node[above, font=\scriptsize] {$\delta$} (Q);

\end{tikzpicture}
\end{equation}

where $\za\zb$ and $\zb\zg$ in the ideal of the relations.
Furthermore,  $\omega_{\eta'}(s)=|\za|+|\zb|-1$ and $\omega_{\eta'}(t)=|\zb|+|\zg|-1$, thus we may assume that $|\za|=|\zg|=0$ and $|\zb|=-m+1$ by shifting the grading if necessary. 
In this case, $e_1A'\oplus e_2A'$ is a pre-silting object in $\per(A')$, since $\Hom_{\per(A')}(e_iA',e_j A'[n])=\h^n(e_jA'e_i)=0$ for any $i=1,2$ and $n>0$.

Assume that $\per(A)$ is silting-discrete. 
Let $F: \per(A')\rightarrow \per(A)$ be the derived equivalence induced by the homeomorphism between $(S,\eta)$ and $(S',\eta')$ (see \cite[Corollary 1.10]{LP20}). Let $P:=F(e_1A'\oplus e_2A')$. Since $e_1A'\oplus e_2A'$ is a pre-silting object in $\per(A')$, $P$ is pre-silting in $\per(A)$, and thus partial silting by Proposition \ref{Prop:silt}. Therefore, there exists a pre-silting object $P'$ in $\per(A)$ such that $P\oplus P'$ is silting.

Note that $\End_{\per(A)}(P) \cong \End_{\per(A')}(e_1A' \oplus e_2A') \cong H^0((e_1+e_2)A'(e_1+e_2))$. 
Since this algebra contains a Kronecker quiver, it admits a band $\alpha\gamma^{-1}$, 
which implies that $\End_{\per(A)}(P \oplus P')$ is representation-infinite. 
It follows from \cite{P19} that $\End_{\per(A)}(P \oplus P')$ is $\tau$-tilting infinite. 
However, this contradicts Proposition \ref{Prop:equ-silting-discrete}, 
and we therefore conclude that $\per(A)$ is not silting-discrete.
\end{proof}

We next show the second part of the necessary condition in Theorem \ref{thm:main1}.

\begin{lemma}\label{lem:key2}
  Assume $S$ is of genus $0$ and $\per(A)$ is silting-discrete. Then $\omega_\eta(\ell)\not=0$ for any simple closed curve $\ell$ on $S$.
\end{lemma}
\begin{proof}
Assume there exists a simple closed curve $\ell$ whose winding number is $0$. 
Consider the curves  $\{1,2\}$ as shown in the figure below; that is, the smoothing of curves $1$ and $2$ is isotopic to $\ell$. 
Then the curves $1$ and $2$ form an admissible collection on the surface and can be extended to an admissible dissection. The associated graded gentle algebra is denoted by $A'=\Bbbk Q'/\langle I'\rangle$.
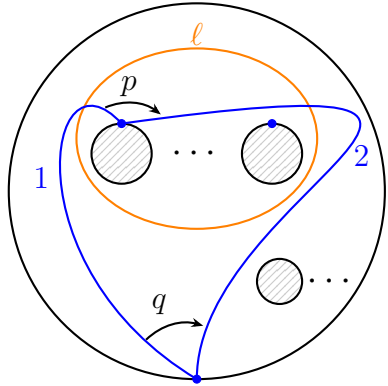
\begin{center}
\begin{tikzpicture}[scale=1, >=Stealth]

\useasboundingbox (-2.8,-2.8) rectangle (2.8,2.8);

\tikzset{hole/.style={thick, draw=black, pattern=north east lines, pattern color=gray!40}}
\tikzset{angleArrow/.style={-{Stealth[width=3.5pt, length=4.5pt]}, thick, black, shorten >=1pt, shorten <=1pt}}

\draw[thick] (0,0) circle (2.5cm); 

\begin{scope}[yshift=0.5cm]
    \foreach \x in {-1,1} {
        \draw[hole] (\x,0) circle (0.4cm);
        \filldraw[blue] (\x,0.4) circle (1.5pt);
    }
    \node[scale=1.2] at (0,0) {$\cdots$};
    \draw[orange, thick] (0,0.2) ellipse (1.6cm and 1.2cm) node[above=1.1cm]{$\ell$};
\end{scope}

\begin{scope}[yshift=-1.2cm]
    \draw[hole] (1.1,0) circle (0.3cm);
    \filldraw[blue] (0,-1.3) circle (1.5pt);
    \node[scale=1.2] at (1.8,0) {$\cdots$};
\end{scope}

\begin{scope}
    \draw[blue, thick]
        (-1,0.9) .. controls (-2,2) and (-2.5,-1) ..
        (0,-2.5)
        node[midway, left, blue] {1};
    \draw[blue, thick]
        (-1,0.9) .. controls (5.5,1.8) and (0,0) ..
        (0,-2.5)
        node[midway, right, blue] {2};

    \draw[angleArrow, thick, bend left=30] (-.7,-2) to (0.12,-1.8);
    \node at (-.5,-1.5) {$q$};
    \draw[angleArrow, thick, bend left=30] (-1.25,1.1) to (-.45,1);
    \node at (-0.9,1.4) {$p$};
\end{scope}
\end{tikzpicture}
\end{center}

Let $p$ and $q$ be the two paths in $Q'$ from vertex $1$ to $2$, as illustrated in the figure above. 
Given the assumption that $\omega_\eta(\ell)=0$ and the fact that $\omega_\eta(\ell)=|p|-|q|$ 
(cf. \cite[Section 3.2]{JSW23}), we obtain $|p|=|q|$. 
Furthermore, by applying a grading shift to $A'$, we may assume without loss of generality 
that $|p|=|q|=0$.
Note that we can read any path in the quiver from the associated surface, and by the picture depicted above, it is clear that any possible path starting at vertex $2$ and ending at vertex $1$ must lie in the ideal $I'$. Thus we have
\[\Hom_{\per(A')}(e_1A'\oplus e_2A',(e_1A'\oplus e_2A')[n])=\h^n(e_1A'e_1)\oplus \h^n(e_1A'e_2)\op\h^n(e_2A'e_2)=0, \]
for any $n>0$, which means that $e_1A'\oplus e_2A'$ is a pre-silting object in $\per(A')$. 
Since $\per(A')\simeq \per(A)$ is silting-discrete, by Proposition \ref{Prop:silt}, $e_1A'\oplus e_2A'$ can be completed to a silting object $P$. 

Note that $\h^0((e_1+e_2)A'(e_1+e_2))$ is the algebra defined by the Kronecker quiver, which is well-known to be representation-infinite and thus $\tau$-tilting infinite. Since this algebra is a subalgebra of $\End_{\per(A')}(P)$ given by an idempotent, it follows that  $\End_{\per(A')}(P)$ is also $\tau$-tilting infinite (see for example \cite[Corollary 2.4]{P19}). This contradicts the silting-discreteness of $\per(A')$ by Proposition \ref{Prop:equ-silting-discrete}.
Therefore, $\omega_\eta(\ell)\not=0$ for any simple closed curve $\ell$ on $S$.
\end{proof}
We remark that Lemma \ref{lem:key2} does not hold for arbitrary closed curves. More precisely, if $\per(A)$ is silting-discrete, there may exist a closed curve $\ell$ whose winding number is zero, as illustrated in the following example. 
\begin{example}
Let $A=\Bbbk Q/\langle I\rangle$ be a graded gentle algebra given by the following quiver
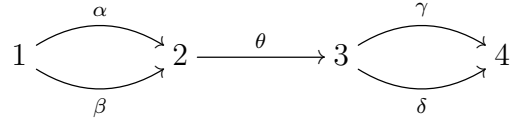
\[
\xymatrix@C=4pc{ 
1\ar@/^1pc/[r]^{\alpha}       
\ar@/_1pc/[r]_{\beta}="B"     &
2 \ar[r]^{\theta}="T"         &
3 \ar@/^1pc/[r]^{\gamma}        
\ar@/_1pc/[r]_{\delta}="D"    &
4  } \]
with grading  $|\alpha|=1=|\gamma|$, $|\beta|=0=|\delta|$, and $|\theta|=0$. 
 And $I$ is the set of relations given by $\beta\theta$ and $\theta\delta$.
The associated surface model $(S,\eta,\Delta)$ is a disk with $3$ boundary components, 
see the left picture in the following

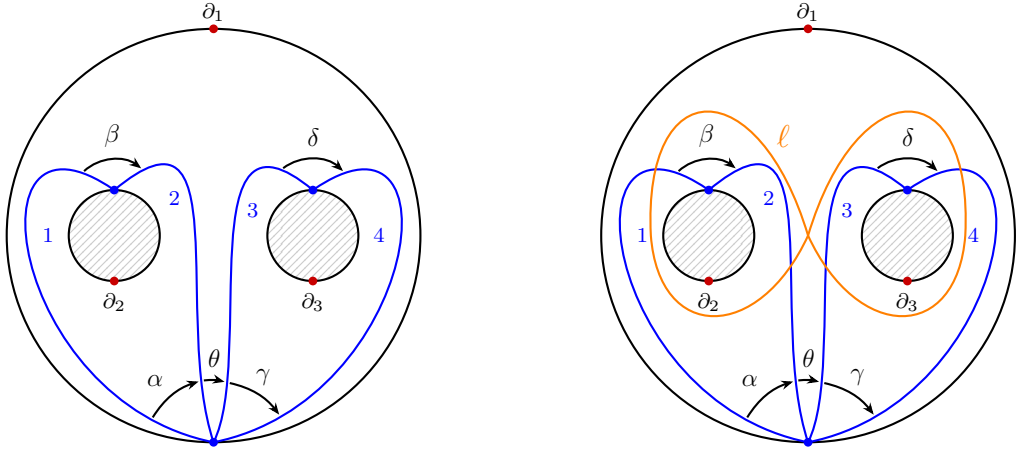
\begin{center}
\begin{tikzpicture}[scale=1.1, >=Stealth]

\tikzset{angleArrow/.style={-{Stealth[width=3.5pt, length=4.5pt]}, thick, black, shorten >=1pt, shorten <=1pt}}

\coordinate (O) at (0,0);
\def\R{2.5}
\coordinate (C2) at (-1.2, 0);
\coordinate (C3) at (1.2, 0);
\def\rsmall{0.55}

\draw[thick] (O) circle (\R);
\node at (0, \R+0.2) {\tiny${\partial_1}$};

\draw[thick, fill=white, pattern=north east lines, pattern color=gray!40] 
      (C2) circle (\rsmall);
\node at ($(C2)+(0, -\rsmall-0.25)$) {\tiny$\partial_2$};

\draw[thick, fill=white, pattern=north east lines, pattern color=gray!40] 
      (C3) circle (\rsmall);
\node at ($(C3)+(0, -\rsmall-0.25)$) {\tiny$\partial_3$};

\coordinate (G1) at (0, -\R); \coordinate (R1) at (0, \R);
\coordinate (G2) at ($(C2)+(0, \rsmall)$); \coordinate (R2) at ($(C2)+(0, -\rsmall)$);
\coordinate (G3) at ($(C3)+(0, \rsmall)$); \coordinate (R3) at ($(C3)+(0, -\rsmall)$);

\draw[arcblue, thick] (G1) .. controls (-2.5, -2) and (-3.0, 1.8) .. (G2)
    coordinate[pos=0.11] (a1_start) coordinate[pos=0.92] (a1_end);
\node[arcblue] at (-2, 0.0) {\tiny{1}};

\draw[arcblue, thick] (G1) .. controls (-0.5, -1) and (0.3, 1.8) .. (G2)
    coordinate[pos=0.15] (a2_start) coordinate[pos=0.91] (a2_end) coordinate[pos=0.6] (a2_mid);
\node[arcblue, left] at (a2_mid) {\tiny{2}};

\draw[arcblue, thick] (G1) .. controls (0.5, -1.5) and (-0.3, 1.8) .. (G3)
    coordinate[pos=0.18] (a3_start) coordinate[pos=0.9] (a3_end) coordinate[pos=0.6] (a3_mid);
\node[arcblue, right] at (a3_mid) {\tiny{3}};

\draw[arcblue, thick] (G1) .. controls (2.5, -2) and (3.0, 1.8) .. (G3)
    coordinate[pos=0.12] (a4_start) coordinate[pos=0.92] (a4_end);  
\node[arcblue] at (2, 0.0) {\tiny{4}};

\draw[angleArrow, thick, bend left=20] (a1_start) to node[above left, scale=0.8, xshift=0pt]{$\alpha$} (a2_start);
\draw[angleArrow, thick, bend left=40] (a1_end) to node[above, scale=0.8, yshift=2pt]{$\beta$} (a2_end);
\draw[angleArrow, thick, bend left=10] (a2_start) to node[above, scale=0.8, yshift=2pt]{$\theta$} (a3_start);
\draw[angleArrow, thick, bend left=20] (a3_start) to node[above right, scale=0.8, xshift=-5pt]{$\gamma$} (a4_start);
\draw[angleArrow, thick, bend left=40] (a3_end) to node[above, scale=0.8, yshift=2pt]{$\delta$} (a4_end);

\foreach \p in {G1, G2, G3} \fill[arcblue] (\p) circle (1.5pt);
\foreach \p in {R1, R2, R3} \fill[myred] (\p) circle (1.5pt);
\end{tikzpicture}
\hspace{1cm}
%
\begin{tikzpicture}[scale=1.1, >=Stealth]

\coordinate (O) at (0,0);
\def\R{2.5}
\coordinate (C2) at (-1.2, 0);
\coordinate (C3) at (1.2, 0);
\def\rsmall{0.55}

\draw[thick] (O) circle (\R);
\node at (0, \R+0.2) {\tiny${\partial_1}$};

\draw[thick, fill=white, pattern=north east lines, pattern color=gray!40] 
      (C2) circle (\rsmall);
\node at ($(C2)+(0, -\rsmall-0.25)$) {\tiny$\partial_2$};

\draw[thick, fill=white, pattern=north east lines, pattern color=gray!40] 
      (C3) circle (\rsmall);
\node at ($(C3)+(0, -\rsmall-0.25)$) {\tiny$\partial_3$};

\coordinate (G1) at (0, -\R); \coordinate (R1) at (0, \R);
\coordinate (G2) at ($(C2)+(0, \rsmall)$); \coordinate (R2) at ($(C2)+(0, -\rsmall)$);
\coordinate (G3) at ($(C3)+(0, \rsmall)$); \coordinate (R3) at ($(C3)+(0, -\rsmall)$);

\draw[arcblue, thick] (G1) .. controls (-2.5, -2) and (-3.0, 1.8) .. (G2)
    coordinate[pos=0.11] (a1_start) coordinate[pos=0.92] (a1_end);
\node[arcblue] at (-2, 0.0) {\tiny{1}};

\draw[arcblue, thick] (G1) .. controls (-0.5, -1) and (0.3, 1.8) .. (G2)
    coordinate[pos=0.15] (a2_start) coordinate[pos=0.91] (a2_end) coordinate[pos=0.6] (a2_mid);
\node[arcblue, left] at (a2_mid) {\tiny{2}};

\draw[arcblue, thick] (G1) .. controls (0.5, -1.5) and (-0.3, 1.8) .. (G3)
    coordinate[pos=0.18] (a3_start) coordinate[pos=0.9] (a3_end) coordinate[pos=0.6] (a3_mid);
\node[arcblue, right] at (a3_mid) {\tiny{3}};

\draw[arcblue, thick] (G1) .. controls (2.5, -2) and (3.0, 1.8) .. (G3)
    coordinate[pos=0.12] (a4_start) coordinate[pos=0.92] (a4_end);  
\node[arcblue] at (2, 0.0) {\tiny{4}};


\tikzset{angleArrow/.style={-{Stealth[width=3.5pt, length=4.5pt]}, thick, black, shorten >=1pt, shorten <=1pt}}

\draw[angleArrow, thick, bend left=20] (a1_start) to node[above left, scale=0.8, xshift=0pt]{$\alpha$} (a2_start);
\draw[angleArrow, thick, bend left=40] (a1_end) to node[above, scale=0.8, yshift=2pt]{$\beta$} (a2_end);
\draw[angleArrow, thick, bend left=10] (a2_start) to node[above, scale=0.8, yshift=2pt]{$\theta$} (a3_start);
\draw[angleArrow, thick, bend left=20] (a3_start) to node[above right, scale=0.8, xshift=-5pt]{$\gamma$} (a4_start);
\draw[angleArrow, thick, bend left=40] (a3_end) to node[above, scale=0.8, yshift=2pt]{$\delta$} (a4_end);

\foreach \p in {G1, G2, G3} \fill[arcblue] (\p) circle (1.5pt);
\foreach \p in {R1, R2, R3} \fill[myred] (\p) circle (1.5pt);

\draw[thick, orange, smooth] (0,0) 
    .. controls (0.6, 2) and (2, 2) .. (1.9, 0)    
    .. controls (1.8, -1.3) and (0.6, -1.3) .. (0,0)      
    .. controls (-0.6, 2) and (-2, 2) .. (-1.9, 0)
    .. controls (-1.8, -1.3) and (-0.6, -1.3) .. (0,0); 
\node[orange] at (-.3, 1.2) {$\ell$};
\end{tikzpicture}
\end{center}
There are only three simple closed curves on $S$, up to homotopy, which go around each of the three boundary components. We have $\omega_\eta(\partial_2)=|\za|-|\zb|=1$ and $\omega_\eta(\partial_3)=|\zg|-|\zd|=1$. By the Poincar\'e-Hopf index formula, we have $\omega_\eta(\partial_1)=-4$. Thus, the winding number of any simple closed curve is nonzero, and by Theorem \ref{thm:main1}, $\per(A)$ is silting-discrete. 

Now we consider the closed curve $\ell$ shown on the right in the above picture, which is a figure-eight curve encircling the two inner boundary components.
By decomposing it into four segments (\cite[Section 3.2]{JSW23}), we can calculate the winding number of $\ell$ as follows:
\[\omega_\eta(\ell)=(1-|\zb|-|\theta|)+(-|\zg|)+(|\zd|+|\theta|-1)+|\za|=0.\]
Thus, $\ell$ is a closed curve whose winding number is zero even though $\per(A)$ is silting-discrete. 
\end{example}

\subsection{The sufficient part of Theorem \ref{thm:main1}}

In this subsection, we show the sufficient part of the conditions in Theorem \ref{thm:main1}.

\begin{Lem}\label{lem:key3}
If $S$ is of genus zero and $\omega_\eta(\ell)\neq 0$ for any simple closed curve $\ell$, then $\per(A)$ is silting-discrete.
\end{Lem}

\begin{proof}
If $\per(A)$ is not silting-discrete, then by Proposition \ref{Prop:equ-silting-discrete}, there exists a silting object $P\in \per(A)$ such that $\End_{\per(A)}(P)$ is $\tau$-tilting infinite, and hence representation-infinite by \cite{P19}. 
Without loss of generality, we may assume that $A$ itself is such a silting object in $\per(A)$; that is, $A$ is non-positive and $\h^0(A)$ is representation-infinite. 

Since $A$ is non-positive, $\h^0(A)$ is given by the subquiver $\tilde{Q}$ of $Q$ consisting of arrows with grading zero. Since $\h^0(A)$ is representation-infinite, by \cite[Lemma 4.1]{P19}, $\tilde{Q}$, and hence $Q$, either contains a subquiver of the form $\tilde{A}_m$:
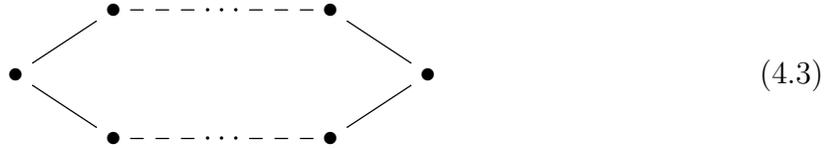
\begin{equation} \label{eq:typeI}
\vcenter{
\xymatrix@C=2pc@R=1pc{
    & \bullet \ar@{--}[r] 
    & \cdots \ar@{--}[r] 
    & \bullet \ar@{-}[dr] 
    & \\
    \bullet \ar@{-}[ur] \ar@{-}[dr] 
    & & & 
    & \bullet \\
    & \bullet \ar@{--}[r] 
    & \cdots \ar@{--}[r] 
    & \bullet \ar@{-}[ur] 
    &
}
}
\end{equation}
or contains a subquiver of the form:
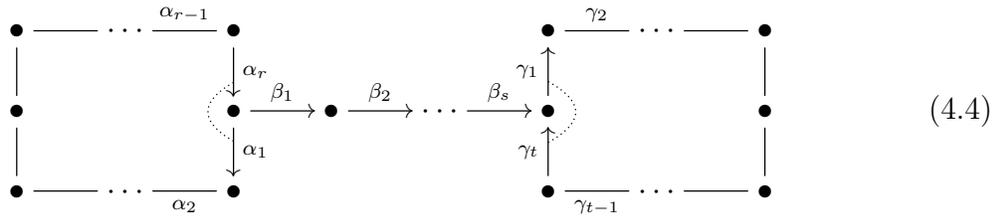
\begin{equation} \label{eq:typeII}
\vcenter{
\xymatrix@C=2pc@R=1.5pc{
    \bullet \ar@{-}[d] \ar@{-}[r] 
    & \cdots \ar@{-}[r]^-{\alpha_{r-1}} 
    & \bullet \ar[d]^-{\alpha_r}="AR" 
    & & & 
    \bullet \ar@{-}[r]^-{\gamma_2} 
    & \cdots \ar@{-}[r] 
    & \bullet \ar@{-}[d] 
    \\
    \bullet \ar@{-}[d] 
    & 
    & \bullet \ar[d]^-{\alpha_1}="A1" \ar[r]^-{\beta_1} 
    & \bullet \ar[r]^-{\beta_2} 
    & \cdots \ar[r]^-{\beta_s} 
    & \bullet \ar[u]^-{\gamma_1}="G1" 
    & 
    & \bullet \ar@{-}[d] 
    \\
    \bullet 
    & \cdots \ar@{-}[l] 
    & \bullet \ar@{-}[l]^-{\alpha_2} 
    & & & 
    \bullet \ar[u]^-{\gamma_t}="GT" 
    & \cdots \ar@{-}[l]^-{\gamma_{t-1}} 
    & \bullet \ar@{-}[l] 
    \ar@{..}@/_1.5pc/ "AR";"A1"
    %
    \ar@{..}@/_1.5pc/ "GT";"G1"
}
}
\end{equation}

\emph{Case 1:} $Q$ contains a subquiver of the form \eqref{eq:typeI}.
We traverse the subquiver $\tilde{A}_m$ in a clockwise direction starting from a source or sink vertex. We decompose the cycle into a concatenation of homotopy letters $\sigma = \sigma_1 \sigma_2 \cdots \sigma_t$, where each $\sigma_i$ is a maximal sequence of consecutive arrows oriented in the same direction (either following or opposing the clockwise orientation).
By construction, the sequence $\{\sigma_i\}_{1 \le i \le t}$ strictly alternates between direct and inverse homotopy letters. Thus, the number of direct homotopy letters equals the number of inverse homotopy letters, and $\sigma$ is a homotopy band (see \cite[Definition 2.1]{OPS18}). This implies that there exists a closed curve $\ell$ on the surface $S$ whose winding number is zero corresponding to $\sigma$. 

We can read $\sigma$ on the surface as follows. Each homotopy letter $\sigma_i$ corresponds to a segment of $\ell$ passing through the polygon(s) cut by the dual dissection $\Delta^*$. Note that $\ell$ is simple, because otherwise $\ell$ would have a self-intersection point, as shown in the following figure. Then the two corresponding homotopy letters $\sigma_i$ and $\sigma_j$ share at least one common arrow, which contradicts to the fact that  $\sigma$ is given by a quiver of type $\tilde{A}_m$. Thus $\ell$ is a simple closed curve whose winding number is zero, contradicting our assumption.

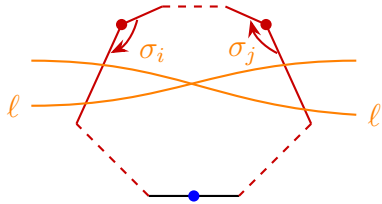
\begin{center}
\begin{tikzpicture}[scale=1.2, >=Stealth]

    \definecolor{arcblue}{RGB}{0, 0, 255}
    \definecolor{myred}{RGB}{200, 0, 0}

    \tikzset{
        dot/.style={circle, fill, inner sep=1.5pt},
        redDot/.style={dot, fill=myred},
        blueDot/.style={dot, fill=arcblue},
        dashedEdge/.style={thick, dashed, myred},
        solidEdge/.style={thick},
        angleArrow/.style={<-, thick, myred, shorten >=1pt, shorten <=1pt},
        gammaCurve/.style={thick, orange, smooth}
    }

    \coordinate (TL) at (-0.8, 1.5);
    \coordinate (TR) at (0.8, 1.5);
    \coordinate (Top_L) at (-0.35, 1.7);
    \coordinate (Top_R) at (0.35, 1.7);
    \coordinate (ML) at (-1.3, 0.4);
    \coordinate (MR) at (1.3, 0.4);
    \coordinate (BL) at (-0.5, -0.4);
    \coordinate (BR) at (0.5, -0.4);
    \coordinate (BMid) at ($(BL)!0.5!(BR)$);

    \draw[solidEdge, myred] (ML) -- (TL);
    \draw[solidEdge, myred] (MR) -- (TR);
    
    \draw[solidEdge, myred] (TL) -- (Top_L);
    \draw[dashedEdge] (Top_L) -- (Top_R);
    \draw[solidEdge, myred] (Top_R) -- (TR);
    
    \draw[solidEdge] (BR) -- (BL);
    \draw[dashedEdge] (MR) -- (BR);
    \draw[dashedEdge] (BL) -- (ML);

    \node[redDot] at (TL) {};
    \node[redDot] at (TR) {};
    \node[blueDot] at (BMid) {};

    \def\posRatio{0.3}
    
    \coordinate (Arr_L_Side) at ($(TL)!\posRatio!(ML)$);
    \coordinate (Arr_L_Top)  at ($(TL)!\posRatio + 0.1!(Top_L)$);
    \draw[angleArrow] (Arr_L_Side) to[bend right=25] (Arr_L_Top);

    \coordinate (Arr_R_Top)  at ($(TR)!\posRatio + 0.1!(Top_R)$);
    \coordinate (Arr_R_Side) at ($(TR)!\posRatio!(MR)$);
    \draw[angleArrow] (Arr_R_Top) to[bend right=25] (Arr_R_Side);

    \draw[gammaCurve] (-1.8, 1.1) .. controls (-0.2, 1.1) and (0.2, 0.6) .. (1.8, 0.5)
        node[pos=0.35, above] {$\sigma_i$}
        node[right] {$\ell$};

    \draw[gammaCurve] (1.8, 1.1) .. controls (0.2, 1.1) and (-0.2, 0.6) .. (-1.8, 0.6)
        node[pos=0.31, above,yshift=-2pt,xshift=-1.5pt] {$\sigma_j$}
        node[left] {$\ell$};

\end{tikzpicture}
\end{center}

\emph{Case 2}: $Q$ contains a subquiver of the form \eqref{eq:typeII}. 
Let $w_l$ be the string corresponding to the cycle on the left starting at $\za_1$ and ending at $\za_r$. Similarly, let $w_r$ be the string corresponding to the cycle on the right starting at $\zg_1$ and ending at $\zg_t$.
Let $w:=\zb_1\zb_2\cdots \zb_s$ be the string corresponding to the path in the middle. Then we get a generalized walk $w_l w w_rw^{-1}$ from the vertex $s(\za_1)$ to itself. 

Observe that for any direct (resp. inverse) subwalk $\tilde{w}$ of $w_l w w_rw^{-1}$, we have $\tilde{w} \not\in I$ (resp. $\tilde{w}^{-1}\not\in I$) by construction. Thus, we can decompose $w_l w w_rw^{-1}$ into a concatenation of homotopy letters $w_l w w_rw^{-1} = \sigma_1 \sigma_2 \cdots \sigma_m$, such that $\sigma_1,\dots, \sigma_m$ strictly alternate between direct and inverse. Therefore, the number of direct homotopy letters equals the number of inverse homotopy letters, and thus $w_l w w_rw^{-1}$ is a homotopy band.
This implies that there exists a closed curve $\ell$ on the surface $S$ whose winding number is zero corresponding to $w_l w w_rw^{-1}$. 
In fact, by the correspondence between homotopy letters of $w_l w w_rw^{-1}$ and segments of $\ell$, we see that $\ell$ is a simple closed curve, as shown in the following picture, note that here we use the $\rpoint$-dissection, therefore the direction of the arrows is opposite to the direction of the arrows in \eqref{eq:typeII}. This also contradicts our assumption. 
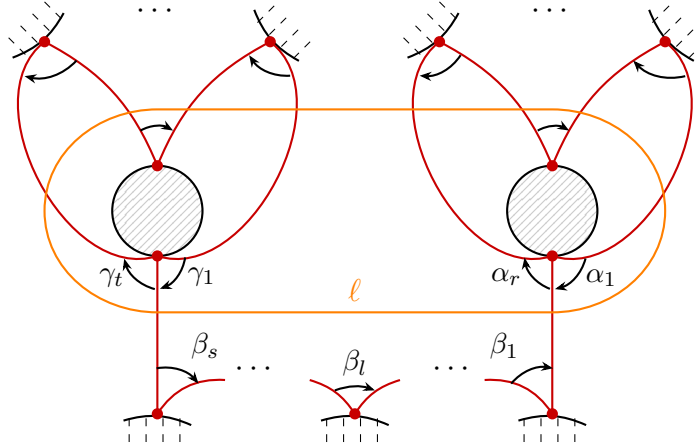
\begin{center}
\begin{tikzpicture}[scale=1.5, >=Stealth]
    \def\GraphDist{3.5} 

    \definecolor{myred}{RGB}{200, 0, 0}
    
    \tikzset{redDot/.style={circle, fill=myred, inner sep=1.5pt}}
    \tikzset{redLine/.style={thick, myred}}
    \tikzset{hole/.style={thick, draw=black, fill=white, pattern=north east lines, pattern color=gray!40}}
        
\tikzset{angleArrow/.style={-{Stealth[width=3.5pt, length=4.5pt]}, thick, black, shorten >=1pt, shorten <=1pt}}
    
    \def\BoundaryMark#1#2{
        \begin{scope}[shift={(#1)}, rotate=#2]
            \draw[thick, black] (-0.3, 0.1) to[bend right=20] (0.3, 0.06);
            \foreach \x in {-0.2, -0.05, 0.1, 0.25} {
                \draw[thin, dashed] (\x, 0.04) -- (\x, 0.25);
            }
        \end{scope}
    }

    \def\DrawGraphUnit#1#2{
        \begin{scope}[shift={(#2, 0)}]
            \coordinate (#1-O) at (0,0);
            \def\R{0.4}
            \coordinate (#1-P3) at (0, \R);
            \coordinate (#1-P4) at (0, -\R);
            \coordinate (#1-P1) at (1, 1.5);
            \coordinate (#1-P2) at (-1, 1.5);
            \coordinate (#1-P5) at (0, -1.8);

            \BoundaryMark{#1-P1}{-45}
            \BoundaryMark{#1-P2}{45}
            \BoundaryMark{#1-P5}{180}
            \draw[hole] (#1-O) circle (\R);
            \draw[redLine] (#1-P3) to[bend left=20] (#1-P1);
            \draw[redLine] (#1-P3) to[bend right=20] (#1-P2);
            \draw[redLine] (#1-P4) to[bend right=80] (#1-P1);
            \draw[redLine] (#1-P4) to[bend left=80] (#1-P2);
            \draw[redLine] (#1-P4) to[bend right=0] (#1-P5);
            
            \path (#1-P2) -- (#1-P1) node[midway, above=5pt, font=\bfseries] {$\cdots$};
        \end{scope}
    }

    \DrawGraphUnit{L}{0}
    \DrawGraphUnit{R}{\GraphDist}

    \coordinate (P6) at ($(L-P5)!0.5!(R-P5)$);
    \BoundaryMark{P6}{180}

    \draw[redLine] (L-P5) to[bend left=30] coordinate[midway] (L_arc_mid) ($(L-P5)+(0.6, 0.3)$);
    \draw[redLine] (R-P5) to[bend right=30] coordinate[midway] (R_arc_mid) ($(R-P5)+(-0.6, 0.3)$);
    \draw[redLine] (P6) to[bend right=20] coordinate[midway] (P6_left_mid) ($(P6)+(-0.4, 0.3)$);
    \draw[redLine] (P6) to[bend left=20] coordinate[midway] (P6_right_mid) ($(P6)+(0.4, 0.3)$);

    \coordinate (Ellipsis_L_Pos) at ($(L-P5)!0.5!(P6) + (0, 0.4)$);
    \node at (Ellipsis_L_Pos) {$\cdots$};
    \coordinate (Ellipsis_R_Pos) at ($(P6)!0.5!(R-P5) + (0, 0.4)$);
    \node at (Ellipsis_R_Pos) {$\cdots$};

    
    \coordinate (R4_dir1) at ($(R-P4)+(0.3, 0)$);  
    \coordinate (R4_dir5) at ($(R-P4)+(0, -0.3)$); 
    \coordinate (R4_dir2) at ($(R-P4)+(-0.25, 0.01)$); 
    
    \draw[angleArrow] (R4_dir1) to[bend left=30] node[right, text=black, font=\small] {$\alpha_1$} (R4_dir5);
    \draw[angleArrow] (R4_dir5) to[bend left=30] node[left, text=black, font=\small] {$\alpha_r$} (R4_dir2);

    \coordinate (R5_dirNew) at ($(R-P5)+(-0.38, 0.25)$); 
    \coordinate (R5_dir4) at ($(R-P5)+(0.03, 0.4)$);     
    \draw[angleArrow] (R5_dirNew) to[bend left=30] node[above left, text=black, font=\small] {$\beta_1$} (R5_dir4);

    \coordinate (R1_dir4) at ($(R-P1)+(0.2, -0.3)$); 
    \coordinate (R1_dir3) at ($(R-P1)+(-0.3, -0.15)$); 
    \draw[angleArrow] (R1_dir4) to[bend left=30] (R1_dir3);

    \coordinate (R3_dir2) at ($(R-P3)+(-0.15, 0.3)$); 
    \coordinate (R3_dir1) at ($(R-P3)+(0.17, 0.3)$);  
    \draw[angleArrow] (R3_dir2) to[bend left=30] (R3_dir1);

    \coordinate (R2_dir3) at ($(R-P2)+(0.2, -0.1)$);   
    \coordinate (R2_dir4) at ($(R-P2)+(-0.2, -0.3)$);  
    \draw[angleArrow] (R2_dir3) to[bend left=30] (R2_dir4);

    \coordinate (L4_dir1) at ($(L-P4)+(-0.3, 0)$);  
    \coordinate (L4_dir5) at ($(L-P4)+(0, -0.3)$); 
    \coordinate (L4_dir2) at ($(L-P4)+(0.25, 0.01)$); 

    \draw[angleArrow] (L4_dir5) to[bend left=30] node[left, text=black, font=\small] {$\gamma_t$} (L4_dir1);
    \draw[angleArrow] (L4_dir2) to[bend left=30] node[right, text=black, font=\small] {$\gamma_1$} (L4_dir5);

    \coordinate (L5_dirNew) at ($(L-P5)+(0.38, 0.25)$); 
    \coordinate (L5_dir4) at ($(L-P5)+(-0.03, 0.4)$);   
    \draw[angleArrow] (L5_dir4) to[bend left=30] node[above right, text=black, font=\small] {$\beta_s$} (L5_dirNew);

    \coordinate (L2_dir4) at ($(L-P2)+(-0.2, -0.3)$);  
    \coordinate (L2_dir3) at ($(L-P2)+(0.3, -0.15)$);  
    \draw[angleArrow] (L2_dir3) to[bend left=30] (L2_dir4);

    \coordinate (L3_dir2) at ($(L-P3)+(0.16, 0.3)$);   
    \coordinate (L3_dir1) at ($(L-P3)+(-0.17, 0.3)$);   
    \draw[angleArrow] (L3_dir1) to[bend left=30] (L3_dir2);

    \coordinate (L1_dir3) at ($(L-P1)+(-0.2, -0.1)$);  
    \coordinate (L1_dir4) at ($(L-P1)+(0.2, -0.3)$);   
    \draw[angleArrow] (L1_dir4) to[bend left=30] (L1_dir3);

    \coordinate (P6_arrow_start) at ($(P6)+(-0.2, 0.2)$); 
    \coordinate (P6_arrow_end) at ($(P6)+(0.2, 0.2)$);   
    \draw[angleArrow] (P6_arrow_start) to[bend left=20] node[above, text=black, font=\small] {$\beta_l$} (P6_arrow_end);

    \node[redDot] at (L-P3) {}; \node[redDot] at (L-P4) {}; \node[redDot] at (L-P1) {}; \node[redDot] at (L-P2) {}; \node[redDot] at (L-P5) {};
    \node[redDot] at (R-P3) {}; \node[redDot] at (R-P4) {}; \node[redDot] at (R-P1) {}; \node[redDot] at (R-P2) {}; \node[redDot] at (R-P5) {};
    \node[redDot] at (P6) {};

    \draw[thick, orange] 
        ($(L-O)+(0, 0.9)$) -- ($(R-O)+(0, 0.9)$) 
        to[out=0, in=90] ($(R-O)+(1.0, 0)$)      
        to[out=270, in=0] ($(R-O)+(0, -0.9)$)    
        -- node[above, text=orange] {$\ell$} ($(L-O)+(0, -0.9)$)                   
        to[out=180, in=270] ($(L-O)+(-1.0, 0)$)  
        to[out=90, in=180] ($(L-O)+(0, 0.9)$);   

\end{tikzpicture}
\end{center}
So in both cases, we reach a contradiction. Therefore, $\per(A)$ must be silting-discrete.
\end{proof}
\begin{Rem}
If $Q$ contains a subquiver of the form \eqref{eq:typeII}, the generalized walk $w_l w w_r^{-1}w^{-1}$ is also a homotopy band and hence gives another closed curve $\ell'$ with zero winding number. Note that $\ell'$ is not a simple closed curve, as shown in the figure below.
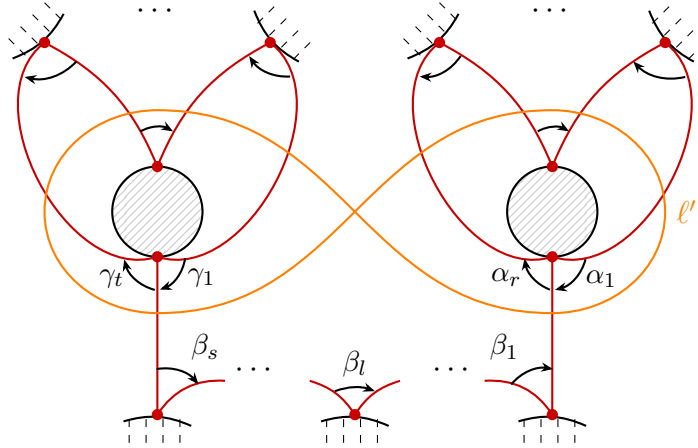
\begin{center}
\begin{tikzpicture}[scale=1.5, >=Stealth]
    \def\GraphDist{3.5} 

    \definecolor{myred}{RGB}{200, 0, 0}
    
    \tikzset{redDot/.style={circle, fill=myred, inner sep=1.5pt}}
    \tikzset{redLine/.style={thick, myred}}
    \tikzset{hole/.style={thick, draw=black, fill=white, pattern=north east lines, pattern color=gray!40}}
        
\tikzset{angleArrow/.style={-{Stealth[width=3.5pt, length=4.5pt]}, thick, black, shorten >=1pt, shorten <=1pt}}
    
    \def\BoundaryMark#1#2{
        \begin{scope}[shift={(#1)}, rotate=#2]
            \draw[thick, black] (-0.3, 0.1) to[bend right=20] (0.3, 0.06);
            \foreach \x in {-0.2, -0.05, 0.1, 0.25} {
                \draw[thin, dashed] (\x, 0.04) -- (\x, 0.25);
            }
        \end{scope}
    }

    \def\DrawGraphUnit#1#2{
        \begin{scope}[shift={(#2, 0)}]
            \coordinate (#1-O) at (0,0);
            \def\R{0.4}
            \coordinate (#1-P3) at (0, \R);
            \coordinate (#1-P4) at (0, -\R);
            \coordinate (#1-P1) at (1, 1.5);
            \coordinate (#1-P2) at (-1, 1.5);
            \coordinate (#1-P5) at (0, -1.8);

            \BoundaryMark{#1-P1}{-45}
            \BoundaryMark{#1-P2}{45}
            \BoundaryMark{#1-P5}{180}
            \draw[hole] (#1-O) circle (\R);
            \draw[redLine] (#1-P3) to[bend left=20] (#1-P1);
            \draw[redLine] (#1-P3) to[bend right=20] (#1-P2);
            \draw[redLine] (#1-P4) to[bend right=80] (#1-P1);
            \draw[redLine] (#1-P4) to[bend left=80] (#1-P2);
            \draw[redLine] (#1-P4) to[bend right=0] (#1-P5);
            
            \path (#1-P2) -- (#1-P1) node[midway, above=5pt, font=\bfseries] {$\cdots$};
        \end{scope}
    }

    \DrawGraphUnit{L}{0}
    \DrawGraphUnit{R}{\GraphDist}

    \coordinate (P6) at ($(L-P5)!0.5!(R-P5)$);
    \BoundaryMark{P6}{180}

    \draw[redLine] (L-P5) to[bend left=30] coordinate[midway] (L_arc_mid) ($(L-P5)+(0.6, 0.3)$);
    \draw[redLine] (R-P5) to[bend right=30] coordinate[midway] (R_arc_mid) ($(R-P5)+(-0.6, 0.3)$);
    \draw[redLine] (P6) to[bend right=20] coordinate[midway] (P6_left_mid) ($(P6)+(-0.4, 0.3)$);
    \draw[redLine] (P6) to[bend left=20] coordinate[midway] (P6_right_mid) ($(P6)+(0.4, 0.3)$);

    \coordinate (Ellipsis_L_Pos) at ($(L-P5)!0.5!(P6) + (0, 0.4)$);
    \node at (Ellipsis_L_Pos) {$\cdots$};
    \coordinate (Ellipsis_R_Pos) at ($(P6)!0.5!(R-P5) + (0, 0.4)$);
    \node at (Ellipsis_R_Pos) {$\cdots$};

    
    \coordinate (R4_dir1) at ($(R-P4)+(0.3, 0)$);  
    \coordinate (R4_dir5) at ($(R-P4)+(0, -0.3)$); 
    \coordinate (R4_dir2) at ($(R-P4)+(-0.25, 0.01)$); 
    
    \draw[angleArrow] (R4_dir1) to[bend left=30] node[right, text=black, font=\small] {$\alpha_1$} (R4_dir5);
    \draw[angleArrow] (R4_dir5) to[bend left=30] node[left, text=black, font=\small] {$\alpha_r$} (R4_dir2);

    \coordinate (R5_dirNew) at ($(R-P5)+(-0.38, 0.25)$); 
    \coordinate (R5_dir4) at ($(R-P5)+(0.03, 0.4)$);     
    \draw[angleArrow] (R5_dirNew) to[bend left=30] node[above left, text=black, font=\small] {$\beta_1$} (R5_dir4);

    \coordinate (R1_dir4) at ($(R-P1)+(0.2, -0.3)$); 
    \coordinate (R1_dir3) at ($(R-P1)+(-0.3, -0.15)$); 
    \draw[angleArrow] (R1_dir4) to[bend left=30] (R1_dir3);

    \coordinate (R3_dir2) at ($(R-P3)+(-0.15, 0.3)$); 
    \coordinate (R3_dir1) at ($(R-P3)+(0.17, 0.3)$);  
    \draw[angleArrow] (R3_dir2) to[bend left=30] (R3_dir1);

    \coordinate (R2_dir3) at ($(R-P2)+(0.2, -0.1)$);   
    \coordinate (R2_dir4) at ($(R-P2)+(-0.2, -0.3)$);  
    \draw[angleArrow] (R2_dir3) to[bend left=30] (R2_dir4);

    \coordinate (L4_dir1) at ($(L-P4)+(-0.3, 0)$);  
    \coordinate (L4_dir5) at ($(L-P4)+(0, -0.3)$); 
    \coordinate (L4_dir2) at ($(L-P4)+(0.25, 0.01)$); 

    \draw[angleArrow] (L4_dir5) to[bend left=30] node[left, text=black, font=\small] {$\gamma_t$} (L4_dir1);
    \draw[angleArrow] (L4_dir2) to[bend left=30] node[right, text=black, font=\small] {$\gamma_1$} (L4_dir5);

    \coordinate (L5_dirNew) at ($(L-P5)+(0.38, 0.25)$); 
    \coordinate (L5_dir4) at ($(L-P5)+(-0.03, 0.4)$);   
    \draw[angleArrow] (L5_dir4) to[bend left=30] node[above right, text=black, font=\small] {$\beta_s$} (L5_dirNew);

    \coordinate (L2_dir4) at ($(L-P2)+(-0.2, -0.3)$);  
    \coordinate (L2_dir3) at ($(L-P2)+(0.3, -0.15)$);  
    \draw[angleArrow] (L2_dir3) to[bend left=30] (L2_dir4);

    \coordinate (L3_dir2) at ($(L-P3)+(0.16, 0.3)$);   
    \coordinate (L3_dir1) at ($(L-P3)+(-0.17, 0.3)$);   
    \draw[angleArrow] (L3_dir1) to[bend left=30] (L3_dir2);

    \coordinate (L1_dir3) at ($(L-P1)+(-0.2, -0.1)$);  
    \coordinate (L1_dir4) at ($(L-P1)+(0.2, -0.3)$);   
    \draw[angleArrow] (L1_dir4) to[bend left=30] (L1_dir3);

    \coordinate (P6_arrow_start) at ($(P6)+(-0.2, 0.2)$); 
    \coordinate (P6_arrow_end) at ($(P6)+(0.2, 0.2)$);   
    \draw[angleArrow] (P6_arrow_start) to[bend left=20] node[above, text=black, font=\small] {$\beta_l$} (P6_arrow_end);

    \node[redDot] at (L-P3) {}; \node[redDot] at (L-P4) {}; \node[redDot] at (L-P1) {}; \node[redDot] at (L-P2) {}; \node[redDot] at (L-P5) {};
    \node[redDot] at (R-P3) {}; \node[redDot] at (R-P4) {}; \node[redDot] at (R-P1) {}; \node[redDot] at (R-P2) {}; \node[redDot] at (R-P5) {};
    \node[redDot] at (P6) {};

    \coordinate (Center) at ($(L-O)!0.5!(R-O)$);

    \draw[thick, orange] 
        ($(L-O)+(0, 0.9)$)                      
        to[out=0, in=135] (Center)              
        to[out=-45, in=180] ($(R-O)+(0, -0.9)$) 

        to[out=0, in=270] ($(R-O)+(1.0, 0)$) node[right, text=orange] {$\ell'$} 
        to[out=90, in=0] ($(R-O)+(0, 0.9)$)     

        to[out=180, in=45] (Center)             
        to[out=225, in=0] ($(L-O)+(0, -0.9)$)   

        to[out=180, in=270] ($(L-O)+(-1.0, 0)$) 
        to[out=90, in=180] ($(L-O)+(0, 0.9)$);  

\end{tikzpicture}
\end{center}
\end{Rem}

\begin{proof}[Proof of Theorem \ref{thm:main1}] Now our result follows directly from Lemmas \ref{lem:key1}, \ref{lem:key2} and \ref{lem:key3}.  
\end{proof}

Assume that the graded surface model $(S,\zD,\eta)$ of $A$ has $b$ boundary components $\partial_1, \partial_2, \ldots, \partial_b$. 
Let $m_i=\omega_\eta(\partial_i)+2$ for each $1\le i \le b$. Then by the Poincar\'e-Hopf index formula \eqref{eq:PH}, we have
\begin{equation*}
    \sum_{i=1}^b m_i = 4-4g. 
\end{equation*}
In particular, if $S$ is of genus $0$, then $\sum_{i=1}^b m_i=4$ always holds. 
We show the following corollary, which states that the silting-discreteness of $\per(A)$ can be characterized by the numbers $m_i$.
\begin{corollary}\label{Cor:siltingdiscretegentle}
    Let $A$ be a homologically smooth, proper graded gentle algebra and let $(S,\eta,\Delta)$ be the associated graded surface. Then the following are equivalent.
    \begin{enumerate}
        \item $\per(A)$ is silting-discrete.
        \item $S$ is of genus $0$ and $m_i$ defined above satisfies  the `No Equipartition' condition
        \[ \sum_{j\in J}m_j\not=\frac{1}{2}\sum_{j=1}^bm_j  \] 
        for all $J\subset\{1,2,\ldots,b\}$.
    \end{enumerate}
\end{corollary}
\begin{proof}
Assume that $A$ is silting-discrete. By Theorem \ref{thm:main1} the genus of $S$ is zero, which is the first part of assumptions in (2). For the second part, let $\ell$ be a simple closed curve on $S$ that partitions the surface into two disjoint regions. Let $S' \subset S$ be the connected component bounded by $\ell$ and the boundary components $\{\partial_{i_1}, \partial_{i_2}, \ldots, \partial_{i_k}\}$. By viewing $S'$ as a stand-alone surface with boundary $\partial S' = \ell \cup \bigcup_{j=1}^k \partial_{i_j}$, we apply the Poincaré-Hopf index formula to $S'$. This yields:
\[
\omega_\eta(\ell) + \sum_{j=1}^k \omega_\eta(\partial_{i_j})  = 4 -2 (k+1).
\]

   Since $m_i=\omega_\eta(\partial_i)+2$, we have
   \[\omega_\eta(\ell)+\sum_{j=1}^k m_{i_j}=2.\]
    Thus, $\omega_\eta(\ell)=0$ if and only if $\sum_{j=1}^k m_{i_j}=2=\frac{1}{2}\sum_{j=1}^bm_j$. 

Conversely, if there exists $J\subset \{1,2,\ldots,b\}$ such that $\sum_{j\in I}m_j=2$, then the simple closed curve $\ell$ that separates the boundary components $\partial_j$ for $j\in J$ from the others satisfies $\omega_\eta(\ell)=0$. The proof is complete.
\end{proof}

\subsection{Silting-discreteness versus derived-discreteness}
In this subsection, we compare silting-discreteness with derived-discreteness for gentle algebras. Let $A=kQ/\langle I\rangle$ be a finite-dimensional gentle algebra. Recall that $A$ is called \emph{derived-discrete} if for any vector $\mathbf{n}=(n_i)_{i\in\Z}$ of natural numbers, there are only finitely many isomorphism classes of indecomposable objects in $\mathcal{D}^{\rm b}(A)$ of cohomology dimension vector $\mathbf{n}$.
Derived-discrete finite-dimensional algebras were classified by Vossieck \cite{V01} and Bobi\'nski-Geiss-Skowro\'nski \cite{BGS04} in terms of certain gentle one-cycle algebras $\Lambda(r,n,m)$ defined as follows. Let $r,n,m$ be non-negative integers with $n\ge r\ge 1$.
Let $Q(r,n,m)$ be the quiver
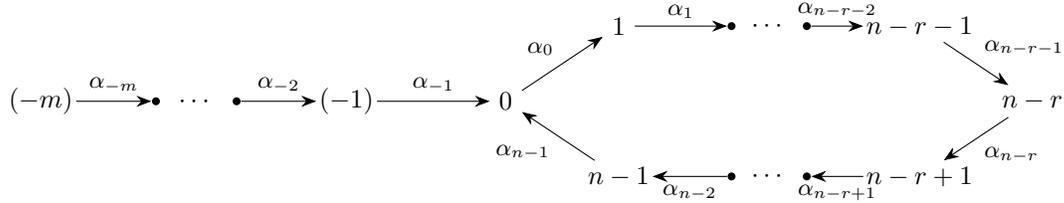
\begin{center}
\begin{tikzpicture}[
    >={Stealth[length=2mm]},
    node distance=1.2cm,
    auto,
    main/.style={inner sep=1pt, minimum size=1em, font=\footnotesize},
    lbl/.style={font=\scriptsize, inner sep=3pt}, 
    bullet/.style={circle, fill, inner sep=0pt, minimum size=3pt} 
]


    \node[main] (0) at (0,0) {$0$};
    
    \node[main] (-1) [left=1.5cm of 0] {$(-1)$};
    
    \node[bullet] (L_b2) [left=1.0cm of -1] {};       
    \node[main]   (L_dots) [left=0.2cm of L_b2] {$\cdots$}; 
    \node[bullet] (L_b1) [left=0.2cm of L_dots] {};   
    
    \node[main] (-m) [left=1.0cm of L_b1] {$(-m)$};

    \node[main] (1) at (1.5, 1.0) {$1$};
    
    \node[bullet] (T_b1) at (3.0, 1.0) {};
    \node[main]   (T_dots) at (3.5, 1.0) {$\cdots$};
    \node[bullet] (T_b2) at (4.0, 1.0) {};
    
    \node[main] (n-r-1) at (5.5, 1.0) {$n-r-1$};
    \node[main] (n-r) at (7.0, 0) {$n-r$};
    \node[main] (n-r+1) at (5.5, -1.0) {$n-r+1$};
    
    \node[bullet] (B_b1) at (4.0, -1.0) {};
    \node[main]   (B_dots) at (3.5, -1.0) {$\cdots$};
    \node[bullet] (B_b2) at (3.0, -1.0) {};
    
    \node[main] (n-1) at (1.5, -1.0) {$n-1$};


    \draw[->] (-m) -- node[lbl, above] {$\alpha_{-m}$} (L_b1);
    \draw[->] (L_b2) -- node[lbl, above] {$\alpha_{-2}$} (-1);
    \draw[->] (-1) -- node[lbl, above] {$\alpha_{-1}$} (0);

    \draw[->] (0) -- node[lbl, above left] {$\alpha_0$} (1);
    \draw[->] (n-1) -- node[lbl, below left] {$\alpha_{n-1}$} (0);

    \draw[->] (1) -- node[lbl, above] {$\alpha_1$} (T_b1);
    \draw[->] (T_b2) -- node[lbl, above] {$\alpha_{n-r-2}$} (n-r-1);

    \draw[->] (n-r-1) -- node[lbl, above right] {$\alpha_{n-r-1}$} (n-r);
    \draw[->] (n-r) -- node[lbl, below right] {$\alpha_{n-r}$} (n-r+1);

    \draw[->] (n-r+1) -- node[lbl, below] {$\alpha_{n-r+1}$} (B_b1);
    \draw[->] (B_b2) -- node[lbl, below] {$\alpha_{n-2}$} (n-1);

\end{tikzpicture}
\end{center}
and let $I$ be the set of relations of $kQ(r,n,m)$ given by the paths $\za_{n-1}\za_{0}$ and $\za_{i}\za_{i+1}$ for all $n-r\le i \le n-2$. 
\begin{Thm}\label{thm:derived-discrete}(\cite[Theorem 2.1]{V01} and \cite[Theorem A]{BGS04})\label{Thm:deriveddiscretegentle}
Let $A$ be a finite-dimensional algebra. Then $A$ is derived-discrete if and only if $A$ is derived equivalent to either a hereditary algebra of Dynkin type or $\Lambda(r,n,m)$ for some integers $m\ge 0$ and $n\ge r\ge 1$.
\end{Thm}
If $\Lambda(r,n,m)$ has finite global dimension, or equivalently, $r$ is strictly smaller than $n$, the surface model of $\Lambda(r,n,m)$ is a graded marked surface $(S,\eta,\Delta)$ with a dissection, where $S$ is an annulus with two boundary components $\partial_1$ and $\partial_2$, with $m+r$ and $n-r$ marked points on $\partial_1$ and $\partial_2$, respectively. See the picture in Figure \ref{fig:derived-discrete}.

\begin{figure}
\begin{center}
\begin{tikzpicture}[scale=1.1, >=Stealth]

    \definecolor{mygreen}{RGB}{0, 180, 0}
    \definecolor{myred}{RGB}{200, 0, 0}
    \definecolor{arcblue}{RGB}{0, 0, 255}

    \tikzset{blueDot/.style={circle, fill=arcblue, inner sep=1.5pt}}
    
    \tikzset{blueLine/.style={thick, arcblue}}
    \tikzset{boundary/.style={thick, black}}
    \tikzset{hole/.style={thick, draw=black, fill=white, pattern=north east lines, pattern color=gray!40}}

\tikzset{angleArrow/.style={-{Stealth[width=3.5pt, length=4.5pt]}, thick, black, shorten >=1pt, shorten <=1pt}}

    \tikzset{gammaCurve/.style={thick, orange}}

    \def\Rout{3.0}
    \def\Rin{0.7}
    \coordinate (CenterIn) at (0, -0.8);

    \coordinate (P1) at (180:\Rout);
    \coordinate (P6) at (0:\Rout);
    \coordinate (P2) at (150:\Rout);
    \coordinate (P3) at (120:\Rout);
    \coordinate (P4) at (90:\Rout);
    \coordinate (P5) at (45:\Rout);
    \coordinate (P8) at (225:\Rout);
    \coordinate (P7) at (315:\Rout);

    \coordinate (P9)  at ($(CenterIn) + (90:\Rin)$);
    \coordinate (P10) at ($(CenterIn) + (180:\Rin)$);
    \coordinate (P11) at ($(CenterIn) + (270:\Rin)$);

    
    \draw[boundary] (0,0) circle (\Rout);
    \node at (45:3.4) {$\partial_1$};

    \draw[hole] (CenterIn) circle (\Rin);
    \node at ($(CenterIn) + (45:0.35)$) {\small $\partial_2$};

    \draw[gammaCurve, postaction={decorate, decoration={markings, mark=at position 0.85 with {\arrow[scale=1.0, orange]{>}}}}] 
        ($(CenterIn)+(0:1.8)$) arc (0:-360:1.7 and 1.3);
    \node[orange, above right] at ($(CenterIn)+(40:1.5)$) {\large $\ell$};

    
    \draw[blueLine] (P6) to[bend right=20] coordinate[pos=0.4] (OnArc6) ($(P6)!0.3!(P7)$);
    \draw[blueLine] (P7) to[bend left=20] coordinate[pos=0.3] (OnArc7) ($(P7)!0.3!(P6)$);
    \node[font=\large, rotate=65] at (337.5:2.6) {$\cdots$};

    \draw[blueLine] (P1) -- coordinate[pos=0.5] (OnLine12) (P2);
    \draw[blueLine] (P1) to[bend right=15] coordinate[pos=0.3] (OnLine13) (P3);
    \draw[blueLine] (P1) to[bend right=20] coordinate[pos=0.25] (OnLine14) (P4);
    \draw[blueLine] (P1) to[bend left=10] coordinate[pos=0.4] (OnLine19) (P9);
    \draw[blueLine] (P1) -- coordinate[pos=0.45] (OnLine110) (P10);
    \draw[blueLine] (P1) to[bend right=20] coordinate[pos=0.3] (OnLine111) (P11);
    \draw[blueLine] (P1) -- coordinate[pos=0.4] (OnLine18) (P8);
    \draw[blueLine] (P8) to[bend right=20] coordinate[pos=0.1] (OnLine87) coordinate[pos=0.9] (OnLine78) (P7);
    \path (P8) -- coordinate[pos=0.1] (OnLine81) (P1);
    \draw[blueLine] (P4) to[bend right=20] coordinate[pos=0.2] (OnLine45) (P5);
    \path (P4) to[bend left=20] coordinate[pos=0.1] (OnLine41) (P1);
    \draw[blueLine] (P5) to[bend right=20] coordinate[pos=0.2] (OnLine56) coordinate[pos=0.85] (OnLine65) (P6);
    \path (P5) to[bend left=20] coordinate[pos=0.2] (OnLine54) (P4);

    \draw[angleArrow] (OnLine78) to[bend left=45] (OnArc7);
    \draw[angleArrow] (OnArc6) to[bend left=45] (OnLine65);

    \draw[angleArrow] (OnLine12) to[bend left=20] node[midway,yshift=8pt,xshift=4pt, font=\tiny, inner sep=1pt] {$\alpha_{-m}$} (OnLine13);
    \draw[angleArrow] (OnLine19) to[bend left=20] node[midway,yshift=0pt,xshift=6pt, font=\tiny, inner sep=1pt] {$\alpha_{1}$} (OnLine110);
    \draw[thick, black, dashed] (OnLine13) to[bend left=15] (OnLine14);
    \draw[thick, black, dashed] (OnLine110) to[bend left=15] (OnLine111);
    \draw[angleArrow] (OnLine14) to[bend left=20] node[midway, right, font=\tiny] {$\alpha_{0}$} (OnLine19);
    \draw[angleArrow] (OnLine111) to[bend left=20] node[midway, right,yshift=-6pt,xshift=-3pt, font=\tiny] {$\alpha_{n-r}$} (OnLine18);
    \draw[angleArrow] (OnLine81) to[bend left=45] node[midway, above, xshift=8pt, font=\tiny] {$\alpha_{n-r+1}$} (OnLine87);

    \draw[angleArrow] (OnLine45) to[bend left=45] node[midway, below,xshift=2pt, font=\tiny] {$\alpha_{n-1}$} (OnLine41);
    \draw[angleArrow] (OnLine56) to[bend left=45] node[midway, below left,xshift=2pt, yshift=1pt, font=\tiny] {$\alpha_{n-2}$} (OnLine54);

    \foreach \p in {P1, P2, P3, P4, P5, P6, P7, P8, P9, P10, P11} {
        \node[blueDot] at (\p) {};
    }

\end{tikzpicture}
\end{center}
\begin{center}
\caption{The surface model $(S,\eta,\Delta)$ of the derived-discrete algebra $\Lambda(r,n,m)$.}\label{fig:derived-discrete}
\end{center}
\end{figure}
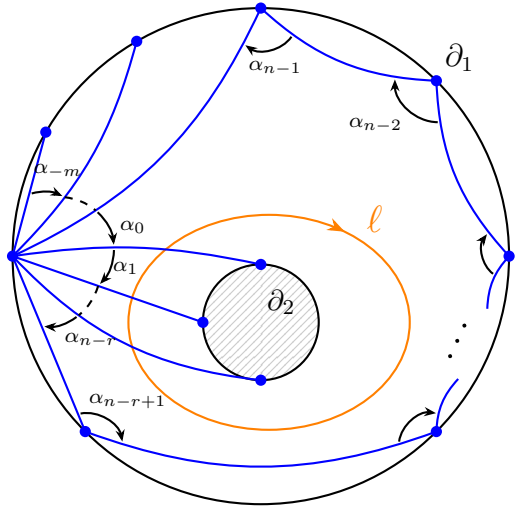

There is only one simple closed curve $\ell$ on $S$ up to homotopy, which winds around the hole once. We have
\[\omega_\eta(\ell)= |\za_0|+|\za_1|+\cdots +|\za_{n-r}| +|\za_{n-r+1}| +\cdots +|\za_{n-1}|-(r+1-1) = -r\not=0,
\]
see for example \cite[Section 3.2]{JSW23} for the calculation.
Then by Theorem \ref{thm:main1} and Theorem \ref{thm:derived-discrete}, we have the following corollary, which is also proved in \cite[Proposition 6.12]{BPP16}, see also \cite[Example 3.9 and Corollary 4.2]{YY21} and \cite[Appendix A]{AH24}.

\begin{corollary}\label{Cor:deriveddiscreteimpliessiltingdiscrete}
If $r<n$ in $A=\Lambda(r,n,m)$, then $\per(A)$ is silting-discrete.
\end{corollary}

\section{Silting-discreteness of graded skew-gentle algebras}

In this section, we investigate the silting-discreteness of graded skew-gentle algebras via their associated surface models. Parallel to the case of graded gentle algebras, every graded skew-gentle algebra admits a geometric realization as a graded marked surface with orbifolds. Within this framework, the set of orbifolds is in one-to-one correspondence with the special loops of the skew-gentle algebra. When the orbifold set is empty, the construction naturally reduces to the surface model of a graded gentle algebra. 

For a comprehensive treatment of surface models for graded skew-gentle algebras, we refer the reader to \cite{AP24, BSW24, CK24,  QZZ22}. Our primary goal in this section is to establish counterparts to Theorem \ref{thm::1.2} and Corollary \ref{Cor:siltingdiscretegentle} for the skew-gentle setting.

Now let $A = \Bbbk Q / \langle I \rangle$ be a graded skew-gentle algebra and let $(S, O, \eta, \Delta)$ denote its associated surface model, where $O$ is the set of orbifolds, $\eta$ is the line field, and $\Delta$ is the dissection. We have the following result.
\begin{theorem}\label{Thm:siltingdiscreteforskewgentle}
Let $A=\Bbbk Q/\langle I \rangle$ be a graded skew-gentle algebra with surface model $(S,O,\eta,\Delta)$. Consider the following conditions.
  \begin{enumerate}
      \item $\per(A)$ is silting-discrete.
      \item $S$ is of genus $0$ and $\# O\le 1$. Moreover, $\omega_\eta(\ell)\neq 0$ for any simple closed curve $\ell$.
      \item $S$ is of genus $0$ and 
        \[ \sum_{j\in J}m_j\not=2 \mbox{ for all $J\subset\{1,2,\ldots,b\}$,} \] 
        where  $m_i=\omega_\eta(\partial_i)+2$ for $1\le i\le b$. 
 \end{enumerate}
 Then (1) implies (2), and (2) is equivalent to (3).
\end{theorem}

We divide the proof of Theorem \ref{Thm:siltingdiscreteforskewgentle} into the following two propositions.

\begin{proposition}\label{prop:Oge2}
      If $\# O\ge 2$, then $\per(A)$ is not silting-discrete. 
\end{proposition}

\begin{proof}
  Assume that $\per(A)$ is silting-discrete. Since $S$ contains at least two orbifolds, by potentially replacing the admissible dissection with one analogous to the construction in \ref{equ:An}, we may assume without loss of generality that the quiver $Q$ takes the following form:
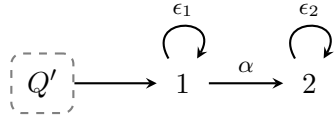
\begin{center}
\begin{tikzpicture}[
    >=stealth,
    node distance=1.2cm,      
    font=\small
]

    \node[draw=gray, dashed, thick, rounded corners=4pt, inner sep=6pt] (Q) {$Q'$};
    \node (1) [right=of Q] {$1$};
    \node (2) [right=of 1] {$2$};

    \draw[->, thick, shorten >=3pt] (Q) -- (1);
    
    \draw[->, thick, shorten >=3pt, shorten <=3pt] (1) -- node[above, font=\scriptsize] {$\alpha$} (2);

    \draw[->, thick, shorten >=2pt, shorten <=2pt] (1) to [out=120, in=60, looseness=6] node[above, font=\scriptsize] {$\epsilon_1$} (1);
    \draw[->, thick, shorten >=2pt, shorten <=2pt] (2) to [out=120, in=60, looseness=6] node[above, font=\scriptsize] {$\epsilon_2$} (2);

\end{tikzpicture}
\end{center}
where $\epsilon_1$ and $\epsilon_2$ are special loops at vertices $1$ and $2$, respectively and $|\za|=0$. Specifically, vertices 1 and 2 constitute the rightmost extremity of $Q$, meaning no additional arrows enter or leave these vertices.
Let $e=e_1+e_2$. Then 
$eA$ is a pre-silting object in $\per(A)$. Thus, there exists $P\in \per(A)$ such that $P\op eA$ is a silting object by Proposition \ref{Prop:silt}. 
We regard $\End(eA)$ as a subalgebra of $\End(eA\op P)$ given by the idempotent ${\rm{id}}_{eA}$. 
Note that by our construction of $Q$, $\End(eA)\cong eAe$ is a skew gentle algebra given by the following sub-quiver of $Q$:
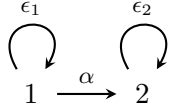
\begin{center}
\begin{tikzpicture}[
    >=stealth,
    node distance=1cm, 
    font=\small
]

    \node (1) {$1$}; 
    \node (2) [right=of 1] {$2$};

    \draw[->, thick, shorten >=3pt, shorten <=3pt] (1) -- node[above, font=\scriptsize] {$\alpha$} (2);

    \draw[->, thick, shorten >=2pt, shorten <=2pt] (1) to [out=120, in=60, looseness=8] node[above, font=\scriptsize] {$\epsilon_1$} (1);
    \draw[->, thick, shorten >=2pt, shorten <=2pt] (2) to [out=120, in=60, looseness=8] node[above, font=\scriptsize] {$\epsilon_2$} (2);

\end{tikzpicture}
\end{center}
It follows that $\End(eA)$ is representation-infinite
as it
contains a minimal symmetric band $\epsilon_1 \alpha \epsilon_2 \alpha^{-1}$ (c.f. Definition \ref{prop::minimal-asymmetric-band}).
According to Theorem \ref{theo::main-result-skew-gentle}, $\End(eA)$ is $\tau$-tilting infinite. Furthermore, by \cite[Corollary 2.4]{P19}, the endomorphism algebra $\End(eA \oplus P)$ must also be $\tau$-tilting infinite. This stands in contradiction to the silting-discreteness of $\per(A)$ guaranteed by Proposition \ref{Prop:equ-silting-discrete}.
\end{proof}

\begin{proposition}\label{prop:Sge1}
    If there exists a simple closed curve on $S$
    with winding number zero, then $\per(A)$ is not silting-discrete.
\end{proposition}
\begin{proof}
Note that if $\# O=0$, then $A$ is a graded gentle algebra and 
the result follows from Theorem \ref{thm:main1}. 
If $\# O\ge 2$, then the result follows from Proposition \ref{prop:Oge2}.
Therefore, let $\# O=1$ in the following.

Assume $\per(A)$ is silting-discrete and $\ell$ is a simple closed curve on $S$ such that $\omega_\eta(\ell)=0$. Note that $\ell$ cannot only surround the orbifold, because
any simple closed curve only surrounding the orbifold point in $O$ in $S$ always has winding number one (see \cite{BSW24}). 
Therefore, there exists a boundary component and 
a simple arc $1$ connecting the orbifold point to  a marked point on that boundary component which does not intersect with $\ell$. 
See the figure below. 

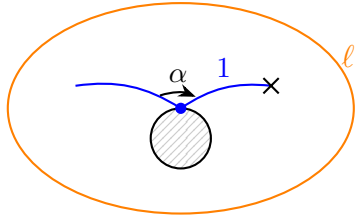
\begin{center}
\begin{tikzpicture}[scale=1, >=Stealth]

\definecolor{arcblue}{RGB}{0, 0, 255}

\tikzset{blueDot/.style={circle, fill=arcblue, inner sep=1.5pt}}

\tikzset{blueLine/.style={thick, arcblue}}
\tikzset{greenLine/.style={thick, arcblue}}
\tikzset{hole/.style={thick, draw=black, fill=white, pattern=north east lines, pattern color=gray!40}}
\tikzset{orbifold/.style={cross out, draw=black, thick, minimum size=5pt, inner sep=0pt}}

\tikzset{angleArrow/.style={->, thick, black, shorten >=1pt, shorten <=1pt}}

\coordinate (HoleCenter) at (0, -0.4);

\def\rsmall{0.4}

\coordinate (EllipseCenter) at (0,0);
\def\Ra{2.3}
\def\Rb{1.4}

\coordinate (GreenP) at ($(HoleCenter) + (0, \rsmall)$);

\coordinate (OrbP) at (1.2, 0.3);


\draw[thick, orange] (EllipseCenter) ellipse ({\Ra} and {\Rb});
\node[orange, right] at ({\Ra*cos(30)}, {\Rb*sin(30)}) {$\ell$};

\draw[hole] (HoleCenter) circle (\rsmall);

\draw[blueLine] (GreenP) to[bend left=20] coordinate[pos=0.2] (BlueArcPoint) node[midway, above, font=\small] {1} (OrbP);

\draw[greenLine] (GreenP) to[bend right=20] coordinate[pos=1.0] (GreenArcEnd) +(-1.4, 0.3);

\draw[angleArrow] ($(GreenP)!0.4!(GreenArcEnd)$) +(0.25, 0.04) to[bend left=20] node[above, text=black, font=\small] {$\alpha$} ($(GreenP)!1!(BlueArcPoint)$);

\node[blueDot] at (GreenP) {};
\node[orbifold] at (OrbP) {};

\end{tikzpicture}
\end{center}
We extend the arc $1$ to an admissible dissection $\Delta'$ of $(S,O,\eta)$ such that the associated algebra is a  graded skew-gentle algebra, denoted by $A'=\Bbbk Q'/\langle I' \rangle$. In addition,  
the quiver $Q'$ has the following form:
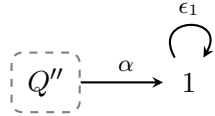
\begin{center}
\begin{tikzpicture}[
    >=stealth,
    node distance=1.2cm,
    font=\small
]

    \node[draw=gray, dashed, thick, rounded corners=4pt, inner sep=6pt] (Q) {$Q''$};
    \node (1) [right=of Q] {$1$};

    \draw[->, thick, shorten >=3pt] (Q) -- node[above, font=\scriptsize] {$\alpha$} (1);

    \draw[->, thick, shorten >=2pt, shorten <=2pt] (1) to [out=120, in=60, looseness=6] node[above, font=\scriptsize] {$\epsilon_1$} (1);

\end{tikzpicture}
\end{center}
where $\epsilon_1$ is the only special loop  corresponding to the unique orbifold in $O$.

Let $Q''$ be the subquiver of $Q'$ obtained by deleting vertex $1$ and the arrows $\epsilon_1$ and $\alpha$. 
Let $I'' = I' \cap \Bbbk Q''$ be the set of relations in $I'$ that are supported on $Q''$, i.e., the paths in $I'$ that do not pass through vertex $1$. 
Let $A'' = \Bbbk Q'' / \langle I'' \rangle$.
Then $A''$ is a graded gentle algebra whose surface model is obtained from $(S,O,\eta,\Delta')$ by removing the orbifold point in $O$ and the arc $1$ in $\zD'$. Note that the genus of the surface remains unchanged.
Therefore, $\ell$ is also a simple closed curve on the surface model of $A''$ with winding number zero  and $\per(A'')$  is not silting-discrete by Theorem \ref{thm:main1}.

Let $e=1-e_1$. Then $A''=eA'e$, and we may regard $\per(A'')$ as a full subcategory of $\per(A')$ via the canonical embedding $\per(A'')\hookrightarrow \per(A')$ given by $-\otimes^{\bf L}_{A''} eA'$. 
Since $\per(A'')$ is not silting-discrete, 
 $\per(A')$ is not silting-discrete by \cite[Theorem 2.10]{AH24}.
 Since $\per(A)$ and $\per(A')$ are triangle equivalent, 
$\per(A)$ is also not silting-discrete. This is a contradiction.
\end{proof}

\begin{proof}[Proof of Theorem \ref{Thm:siltingdiscreteforskewgentle}]
That (1) implies (2) follows from Propositions \ref{prop:Oge2} and \ref{prop:Sge1}. 
The equivalence of (2) and (3) follows from the same argument as in the proof of Corollary \ref{Cor:siltingdiscretegentle}.
\end{proof}

We end our paper with the following remark.

\begin{remark}
In contrast to the case of graded gentle algebras, the class of graded skew-gentle algebras is not necessarily closed under the operation of taking endomorphism algebras of silting objects.
This structural limitation prevents a similar method used in the proof of Theorem \ref{thm::1.1} to establish the implication $(2) \Rightarrow (1)$ via the equivalence between $\tau$-tilting finiteness and representation finiteness for skew-gentle algebras. Despite this technical obstacle in the converse direction, we conjecture that the three conditions in Theorem \ref{Thm:Amain3} remain equivalent.
\end{remark}

\end{document}